\documentclass[journal,onecolumn,draftcls]{IEEEtran}

\usepackage{subfig}
\usepackage{epic,eepic,units}
\usepackage{psfrag}
\usepackage{latexsym}
\usepackage{longtable}
\usepackage{mathrsfs}
\usepackage{bigstrut}
\usepackage{graphicx}
\usepackage{pifont}
\usepackage{amsmath}
\usepackage{setspace}
\usepackage{mathtools}
\usepackage{mathcomp}
\usepackage{mathdots}
\usepackage{rotating}
\usepackage{array}
\usepackage{multirow}
\usepackage{multicol}
\usepackage{color}
\usepackage{url}
\usepackage{algorithm}
\usepackage{algorithmic}
\usepackage{cite,amssymb}
\usepackage[dvipsnames]{xcolor}
\usepackage{parskip}
\usepackage{euscript,yfonts,dsfont,bbm,amstext,wasysym,pdfsync,amsfonts,arydshln,framed}

\usepackage{tikz}
\usetikzlibrary{arrows,shapes,trees,calc,positioning}
\usetikzlibrary{matrix}

\newtheorem{theorem}{Theorem}

\newtheorem{lemma}[theorem]{Lemma}

\newcommand{\enma}[1]   {\ensuremath{#1}}

\newcommand{\beq}{\begin{equation}}
\newcommand{\eeq}{\end{equation}}
\newcommand{\bseq}{\begin{subequations}}
\newcommand{\eseq}{\end{subequations}}
\newcommand{\beqn}{\begin{eqnarray}}
\newcommand{\eeqn}{\end{eqnarray}}
\newcommand{\ba}{\begin{array}}
\newcommand{\ea}{\end{array}}
\newcommand{\bct}{\begin{center}}
\newcommand{\ect}{\end{center}}
\newcommand{\btmz}{\begin{itemize}}
\newcommand{\etmz}{\end{itemize}}
\newcommand{\benum}{\begin{enumerate}}
\newcommand{\eenum}{\end{enumerate}}

\newcommand{\cL}{\enma{\mathcal L}}

\newcommand{\norm}[1]{\| #1 \|}                 

\newcommand{\diag}      {\enma{\mathrm{diag}}}

\newcommand{\trace}     {\enma{\mathrm{trace}}}

\newcommand{\inner}[2]{\left\langle #1,#2 \right\rangle}

\newcommand{\matbegin}{
        \left[
}
\newcommand{\matend}{
        \right]
}

\newcommand{\tbo}[2]{
  \matbegin \begin{array}{c}
       #1 \\ #2
       \end{array} \matend }

\newcommand{\obt}[2]{
  \matbegin \begin{array}{cc}
       #1 & #2
       \end{array} \matend }

\newcommand{\tbt}[4]{
  \matbegin \begin{array}{cc}
       #1 & #2 \\ #3 & #4
       \end{array} \matend }

\newcommand{\be}{\begin{equation}}
\newcommand{\ee}{\end{equation}}

\newcommand{\cplxs}{ C\kern -.35em \rule{0.03 em}{.7 ex}~   }

\def\complex{\hbox{C\kern -.45em \rule{0.03 em}{1.5 ex}}~}

\newcommand{\bi}{\begin{itemize}}
\newcommand{\ei}{\end{itemize}}

\newcommand{\cA}{{\cal A}}
\newcommand{\cl}{{\cal L}}
\newcommand{\cC}{{\cal C}}
\newcommand{\cD}{{\cal D}}

\newcommand{\bE}{{\mathbf{E}}}

\newcommand{\cI}{{\cal I}}

\newcommand{\cB}{{\cal B}}
\newcommand{\cS}{{\cal S}}

\newcommand{\bbR}{\mathbb{R}}
\newcommand{\bbC}{\mathbb{C}}

\newcommand{\eps}{{\epsilon}}

\newcommand{\non}{\nonumber}
\newcommand{\ds}{\displaystyle}

\newcommand{\mrd}{\mathrm{d}}
\newcommand{\mre}{\mathrm{e}}

\newcommand{\DefinedAs}[0]{\mathrel{\mathop:}=}
\DeclareMathOperator*{\argmin}{argmin}
\DeclareMathOperator*{\logdet}{log\,det}

\DeclareMathOperator*{\minimize}{minimize}
\DeclareMathOperator*{\maximize}{maximize}
\DeclareMathOperator*{\subject}{subject~to}
\DeclareMathOperator*{\rank}{rank}

\definecolor{bgblue}{rgb}{0.04,0.19,0.53}
\definecolor{dblue1}{rgb}{0,0.3,0.7}
\definecolor{dred}{rgb}{0.4,0.2,0}

\definecolor{grey}{rgb}{0.6,0.6,0.6}
\definecolor{lightgray}{rgb}{0.97,.99,0.99}

\newcommand{\cT}{\mathcal T}
\newcommand{\In}{\operatorname{In}}
\newtheorem{prop}{Proposition}
\newcommand{\bbH}{\mathbb{H}}

\IEEEoverridecommandlockouts                              
\begin{document}

\title{\LARGE \bf Low-complexity modeling of partially available second-order statistics: theory and an efficient matrix completion algorithm}

\author{Armin Zare, Yongxin Chen, Mihailo R.\ Jovanovi\'c, and Tryphon T. Georgiou
\thanks{Financial support from the National Science Foundation under Award CMMI 1363266, the Air Force Office of Scientific Research under Award FA9550-16-1-0009, the University of Minnesota Informatics Institute Transdisciplinary Faculty Fellowship, and the University of Minnesota Doctoral Dissertation Fellowship is gratefully acknowledged. 
}
\thanks{Armin Zare, Yongxin Chen, Mihailo R.\ Jovanovi\'c and Tryphon T.\ Georgiou are with the Department of Electrical and Computer Engineering, University of Minnesota, Minneapolis, MN 55455. E-mails: arminzare@umn.edu, chen2468@umn.edu, mihailo@umn.edu, tryphon@umn.edu.}
}


\maketitle
	\vspace*{-6ex}
	\begin{abstract}
State statistics of linear systems satisfy certain structural constraints that arise from the underlying dynamics and the directionality of input disturbances.
In the present paper we study the problem of completing partially known state statistics. Our aim is to develop tools that can be used in the context of control-oriented modeling of large-scale dynamical systems. For the type of applications we have in mind, the dynamical interaction between state variables is known while the directionality and dynamics of input excitation is often uncertain. Thus, the goal of the mathematical problem that we formulate is to identify the dynamics and directionality of input excitation in order to explain and complete observed sample statistics. More specifically, we seek to explain correlation data with the least number of possible input disturbance channels. We formulate this inverse problem as rank minimization, and for its solution, we employ a convex relaxation based on the nuclear norm. The resulting optimization problem is cast as a semidefinite program and can be solved using general-purpose solvers.  {For problem sizes that these solvers cannot handle, we develop a customized alternating minimization algorithm (AMA).}
We interpret AMA as a proximal gradient for the dual problem and prove sub-linear convergence for the algorithm with fixed step-size. We conclude with an example that illustrates the utility of our modeling and optimization framework  {and draw contrast between AMA and the commonly used alternating direction method of multipliers (ADMM) algorithm.}
	\end{abstract}
	\vspace*{-0.15cm}

	\vspace*{-3ex}
\section{Introduction}
\label{sec.intro}
Motivation for this work stems from control-oriented modeling of systems with a large number of degrees of freedom. Indeed, dynamics governing many physical systems are prohibitively complex for purposes of control design and optimization. Thus, it is common practice to investigate low-dimensional models that preserve the essential dynamics. To this end, stochastically driven linearized models often represent an effective option that is also capable of explaining observed statistics. Further, such models are well-suited for analysis and synthesis using tools from modern robust control.

An example that illustrates the point is the modeling of fluid flows. In this, the Navier-Stokes equations are prohibitively complex for control design~\cite{kimbew07}. On the other hand,
linearization of the equations around the mean-velocity profile in the presence of stochastic excitation has been shown to qualitatively replicate structural features of shear flows~\cite{farioa93,bamdah01,jovbamACC01,mj-phd04,jovbamjfm05,jovPOF08,moajovJFM10,liemoajovJFM10,moajovJFM12}. However, it has also been recognized that a simple white-in-time stochastic excitation cannot reproduce important statistics of the fluctuating velocity field~\cite{jovbamCDC01,jovgeoAPS10}. In this paper, we {introduce a mathematical framework to} consider {stochastically driven} linear models {that} depart from the white-in-time restriction on {random} disturbances. Our objective is to identify low-complexity disturbance models that account for partially available second-order statistics of large-scale dynamical systems. 

Thus, herein, we formulate a covariance completion problem for linear time-invariant (LTI) systems with uncertain disturbance dynamics. The complexity of the disturbance model is quantified by the number of input channels. We relate the number of input channels to the rank of a certain matrix which reflects the  directionality of input disturbances and the correlation structure of excitation sources. We address the resulting optimization problem using the nuclear norm as a surrogate for rank~\cite{fazhinboy01,faz02,canrec09,kesmonoh10,recfazpar10,canpla10,cantao10,chasanparwil11}.

The relaxed optimization problem is convex {and} can be cast as a semidefinite program (SDP) which is readily solvable by standard software for small-size problems. A further contribution is to specifically address  {larger problems that general-purpose solvers cannot handle.} To this end, we exploit the problem structure, {derive the Lagrange dual,} and develop an efficient customized Alternating Minimization Algorithm (AMA). {Specifically, we show that AMA is a proximal gradient for the dual and establish convergence for the covariance completion problem. We utilize a  {Barzilai-Borwein} (BB) step-size initialization followed by backtracking to achieve sufficient dual ascent. This enhances convergence relative to theoretically-proven sub-linear convergence rates for AMA with fixed step-size. We also draw contrast between AMA and the commonly  used Alternating Direction Method of Multipliers (ADMM) by showing that AMA leads to explicit, easily computable updates of both primal and dual optimization variables.}


The solution to the covariance completion problem gives rise to a class of linear filters that realize colored-in-time disturbances and account for the observed state statistics. This is a non-standard stochastic realization problem with partial spectral information \cite{kal82,geo83,geo87,byrlin94}. The class of modeling filters that we generate for the stochastic excitation is generically minimal in the sense that it has the same number of degrees of freedom as the original linear system.  {Furthermore, we demonstrate that the covariance completion problem can be also interpreted as an identification problem that aims to explain available statistics via suitable low-rank dynamical perturbations.}






Our presentation is organized as follows. We summarize key results regarding the structure of state covariances and its relation to the power spectrum of input processes in Section~\ref{sec.linearstochasticmodels}. We characterize admissible signatures for matrices that parametrize disturbance spectra and formulate the covariance completion problem in Section~\ref{sec.completion}. Section~\ref{sec.algorithms} develops an efficient optimization algorithm for solving this problem in large dimensions. To highlight the theoretical and algorithmic developments we provide an example in Section~\ref{sec.example}. We conclude with remarks and future directions in Section~\ref{sec.conclusion}.

 	 \vspace*{-2ex}
\section{Linear stochastic models and state statistics}
\label{sec.linearstochasticmodels}

We now discuss algebraic conditions that state covariances of LTI systems satisfy. For white-in-time stochastic inputs state statistics satisfy an algebraic Lyapunov equation. A similar algebraic characterization holds for LTI systems driven by colored stochastic processes~\cite{geo02a,geo02b}. This characterization provides the foundation for the covariance completion problem that we study in this paper.

Consider a linear time-invariant system
\be
	\label{eqn:State equation}
	\ba{rcl}
	\dot{x}
	& \!\! = \!\! &
	A \, x \; + \; B \, u
	\\[0.1cm]
	 {
	y }
	& \!\!  {=} \!\! & 
	 {C\, x}
	\ea
\ee
where $x(t)\in\bbC^n$ is a state vector,  {$y(t) \in \bbC^p$ is the output,} and $u(t)\in\mathbb{C}^m$ is a zero-mean stationary stochastic input.  {The dynamic matrix} $A\in\mathbb{C}^{n\times n}$ is Hurwitz, $B\in \mathbb{C}^{n\times m}$ is  {the input} matrix with $m \leq n$, and $(A,B)$ is a controllable pair. Let $X$ be the steady-state covariance of the state vector of system~\eqref{eqn:State equation}, $X=\lim_{t\to\infty}\bE \left( x(t)x^*(t) \right)$, with $\bE$ being the expectation operator. We next review key results and provide new insights into the following questions:
\bi
\item[(i)] What is the algebraic structure of $X$? In other words, given a positive definite matrix $X$, under what conditions does it qualify to be the steady-state covariance of~\eqref{eqn:State equation}?

	\vspace*{0.15cm}

\item[(ii)] Given the steady-state covariance $X$ of~\eqref{eqn:State equation}, what can be said about the power spectra of input processes that are consistent with these statistics?
\ei


	\vspace*{-2ex}
\subsection{Algebraic constraints on admissible covariances}
\label{sec.algebraic_constraints}

The steady-state covariance matrix $X$ of the state vector in~\eqref{eqn:State equation}
satisfies~\cite{geo02a,geo02b}
\begin{subequations}
	\be
	\label{eqn:rank Constraint on Sigma}
	\rank
	\left[
	\begin{matrix}
	AX \,+\, X A^* & B
	\\
	B^* & 0
	\end{matrix}
	\right]
	\, = \;
	\rank \left[\begin{matrix}
	0 & B\\B^* & 0
	\end{matrix}\right].
		\ee
An equivalent characterization is that there is a solution $H\in\mathbb{C}^{n \times m}$ to the equation
	\be
	\label{eqn:Constraint on Sigma}
	A \, X
	\; + \;
	X A^*
	\; = \;
	-B H^* \; - \; H B^*.
	\ee
	\end{subequations}
Either of these conditions, together with the positive definiteness of $X$, completely characterize state covariances of linear dynamical systems driven by white or colored stochastic processes~\cite{geo02a,geo02b}. When the input $u$ is white noise with covariance $W$, $X$ satisfies the algebraic Lyapunov equation
	\be
	\label{eqn:Contraint_whitenoise}
	A \, X \; + \; X A^*
	\; = \;
	-B \, WB^*.
	\non
	\ee
In this case, $H$ in~\eqref{eqn:Constraint on Sigma} is determined by $H= \frac{1}{2} BW$ and the right-hand-side $-B\,WB^*$ is sign-definite.
	In fact, except for this case when the input is white noise, {the matrix $Z$ defined by}
\begin{subequations}
	\label{eqn: Definition of Z}
	\begin{eqnarray}
	Z
	& \!\! \DefinedAs \!\! &
	-\left(AX \,+\, X A^*\right)
	\label{eqn:Constraint left}
	\\
	& \!\! = \!\! &
	BH^* \,+\, HB^*
	\label{eqn:Constraint right}
	\end{eqnarray}
\end{subequations}
may have both positive and negative eigenvalues. Additional discussion on the structure of $Z$ is provided in Section \ref{sec.signature}.

	\vspace*{-2ex}
\subsection{Power spectrum of input process}
\label{sec.filter}

\begin{figure}
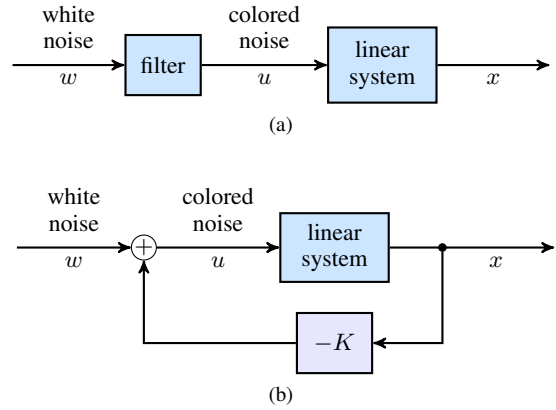

	\begin{center}
	\begin{tabular}{c}
		\subfloat[]{
%
%
%
%
%
%
\input{figures/Tikz_common_styles}
%
%
\noindent
\begin{tikzpicture}[scale=1, auto, >=stealth']
  
    \small

    
     \node[block, minimum height = .8cm, top color=RoyalBlue!20, bottom color=RoyalBlue!20] (sys1) {filter};
     
     \node[block, minimum height = 1cm, top color=RoyalBlue!20, bottom color=RoyalBlue!20] (sys2) at ($(sys1.east) + (2.4cm,0)$) {$\ba{c} \mbox{linear} \\ \mbox{system}\ea$};
     
     \node[] (output-node) at ($(sys2.east) + (1.6cm,0)$) {};
     
     \node[] (input-node) at ($(sys1.west) - (1.6cm,0)$) {}; 
      
%
%
%
%
%
%
%


    \draw [connector] (input-node) -- node [midway, above] {$\ba{c} \mbox{white} \\ \mbox{noise} \ea$} node [midway, below] {$w$} (sys1.west);
    \draw [connector] (sys2.east) -- node [midway, below] {$x$} (output-node);
    \draw [connector] (sys1.east) -- node [midway, above] {$\ba{c} \mbox{colored} \\ \mbox{noise} \ea$} node [midway, below] {$u$} (sys2.west);
\end{tikzpicture}
		         \label{fig.filter}
		         }
		        \\
		\subfloat[]{
%
%
%
%
%
%
\input{figures/Tikz_common_styles}
%
%
\noindent
\begin{tikzpicture}[scale=1, auto, >=stealth']
  
    \small

    
     \node[block, minimum height = .8cm, top color=RoyalBlue!20, bottom color=RoyalBlue!20] (sys1) {$\ba{c} \mbox{linear} \\ \mbox{system}\ea$};
     
     \node[block, minimum height = .8cm, top color=blue!10, bottom color=blue!10] (sys4) at ($(sys1.south) - (0cm,.8cm)$) {$-K$};
     
     \node[] (input-node) at ($(sys1.west) - (3.6cm,0)$) {}; 
     
     \node[] (output-node) at ($(sys1.east) + (2.3cm,0)$) {};
     
     \node[] (mid-node1) at ($(sys1.east) + (1cm,1cm)$) {}; 
          
     \node[sum] (esum1) at ($(sys1.west) - (1.8cm,0)$) {$+$};
           
     \node[branch] (R) at ($(sys1.east) + (.7cm,0.0cm)$){};
     
	

	
    \draw [connector] (input-node) -- node [midway, above] {$\ba{c} \mbox{white} \\ \mbox{noise} \ea$} node [midway, below] {$w$} (esum1.west);
                    	
    \draw [line] (sys1.east) -- (R);
    \draw [connector] (R.west) -- node [midway, below] {$x$} (output-node);
    
    \draw [connector] (R.south) |- (sys4.east);
    
    \draw [connector] (sys4.west) -| (esum1.south);
    
    \draw [connector] (esum1.east) -- node [midway, above] {$\ba{c} \mbox{colored} \\ \mbox{noise} \ea$} node [midway, below] {$u$} (sys1.west);
    
    
\end{tikzpicture}
		         \label{fig.sys-feedback}
		         }
	\end{tabular}
	\end{center}
	\caption{(a) A cascade connection of an LTI system with a linear filter that is designed to account for the sampled steady-state covariance matrix $X$; (b) An equivalent feedback representation of the cascade connection in (a).}
	\label{fig.filter_sys-feedback}
\end{figure}

For stochastically-driven {LTI} systems the state statistics can be obtained from knowledge of the system model and the input statistics. Herein, we are interested in the converse: starting from the steady-state covariance $X$ and the system dynamics \eqref{eqn:State equation}, we want to identify the power spectrum of the input process $u$. As illustrated in Fig.~\ref{fig.filter}, we seek to construct a filter which, when driven by white noise, produces a suitable stationary input  $u$ to~\eqref{eqn:State equation} so that the state covariance is $X$.
Next, we characterize a class of filters with degree at most $n$.

Consider the linear filter given by
\begin{subequations}
	\label{eq:filtermodel}
	\begin{eqnarray}
        \dot{\xi}
        &\!\!=\!\!&
        (A\,-\,BK) \, \xi \; + \; B \, w
        \\[0.cm]
        u
        &\!\!=\!\!&
        -K \, \xi \; + \; w
        \end{eqnarray}
where $w$ is a  {zero-mean} white stochastic process with covariance $\Omega \succ 0$ and
    \begin{equation}
        K \;=\; \frac{1}{2} \, \Omega \, B^*X^{-1} \; - \; H^*X^{-1},
        \label{eq.K}
    \end{equation}
    \end{subequations}
 {for some $H$ that satisfies~\eqref{eqn:Constraint on Sigma}.}
The power spectrum of $u$ is determined by
\be
	\Pi_{uu} (\omega) \;=\; \Psi(j\omega) \, \Omega \, \Psi^* (j\omega)
	\non
\ee
where
\be
    \Psi(s) \;=\; I \; - \; K \, (sI \,-\, A \,+\, BK)^{-1} B
    \non
\ee
is the transfer function of the filter \eqref{eq:filtermodel}.
To verify this, consider the cascade connection shown in Fig.~\ref{fig.filter}, with state space representation
\be
\label{eq.cascade}
	\ba{rcl}
            \left[\begin{array}{c}\dot{x}\\\dot{\xi}\end{array}\right]
            &\!\!=\!\!&
            \left[\begin{array}{cc}A & -BK \\ 0 & A \, - \, BK\end{array}\right]
            \left[\begin{array}{c}x \\ \xi\end{array}\right]
            \,+\,
            \left[\begin{array}{c}B \\ B\end{array}\right]w
            \\
            x
            &\!\!=\!\!&
            \left[\begin{array}{cc}I & 0\end{array}\right]
            \left[\begin{array}{c}x \\ \xi\end{array}\right].
	\ea
\ee
 {This representation has twice as many states as linear system~\eqref{eqn:State equation}, but it is not controllable and therefore not minimal. The coordinate transformation
\[
	\tbo{x}{\phi}
	\;=\;
	\tbt{\phantom{-}I}{0}{-I}{I}
	\tbo{x}{\xi}
\]
brings system~\eqref{eq.cascade} into the following form
\[
	\ba{rcl}
		\tbo{\dot{x}}{\dot{\phi}}
        	&\!\!=\!\!&
        	\tbt{A \,-\, BK}{-BK}{0}{A}\,
        	\tbo{x}{\phi}
		\,+\,
		\tbo{B}{0}
		w
		\\
		x
		&\!\!=\!\!&
		\obt{I}{0}
		\tbo{x}{\phi}.
	\ea
\]
Clearly, the input $w$ does not enter into the equation for $\phi$ and
\be
	\label{eq.compact}
	\dot{x}
	\;=\;
	(A \,-\, BK)\,x
	\,+\,
	B\,w
\ee
provides a minimal realization of the transfer function from white-in-time $w$ to $x$, $(sI-A+BK)^{-1}B$. In addition,} the corresponding algebraic Lyapunov equation in conjunction with~\eqref{eq.K} yields
\begin{align}
	\ba{l}
        \!\!(A\,-\, BK) X \,+\, X (A \,-\, BK)^* \,+\, B \, \Omega \, B^*
        \\[0.1cm]
        \hspace{1cm}
        \;=\;
        AX \,+\, XA^*\,+\, B \, \Omega \, B^* \,-\, BKX \,-\, XK^*B^*
        \\[0.1cm]
        \hspace{1cm}
        \;=\;
        AX \,+\, XA^*\,+\, BH^* \,+\, HB^*
        \\[0.1cm]
        \hspace{1cm}
        \;=\; 0.
        \ea
        \non
\end{align}
This shows that  \eqref{eq:filtermodel} generates a process $u$ that is consistent with $X$.

 {
As we elaborate next, compact representation~\eqref{eq.compact} offers an equivalent interpretation of colored-in-time stochastic input processes as a {\em dynamical perturbation\/} to system~\eqref{eqn:State equation}.
}




	\vspace*{-2ex}
\subsection{Stochastic control interpretation}
\label{sec.feedback_control}

The class of power spectra described by~\eqref{eq:filtermodel} is closely related to the covariance control problem, or the covariance assignment problem, studied in~\cite{HotSke87,CheGeoPav15b}. To illustrate this, let us consider
	\begin{subequations}
\be
	\label{eq:systemwithcontrol}
	\dot{x} \;=\; A \, x \; + \; B \, v \;+\; B \, w
\ee
where $w$ is again white with covariance $\Omega$; see Fig.~\ref{fig.sys-feedback}. In the absence of a control input ($v=0$), the steady-state covariance satisfies the Lyapunov equation
\be
        A \, X \; + \; XA^* \,+\, B \, \Omega \, B^* \;=\; 0.
        \non
\ee
A  choice of a non-zero process $v$ can be used to assign different values for $X$. 
Indeed, for
	\be
		v \; = \; -K \, x
		\label{eq.v}
	\ee
	\end{subequations}
and $A-BK$ Hurwitz, $X$ satisfies
\be
    \label{eq:feedbackcovariance}
        (A\,-\,BK) \, X \;+\; X \, (A\,-\, BK)^* \;+\; B \, \Omega \, B^* \;=\; 0.
\ee
%
It is easy to see that any $X\succ 0$ satisfying~\eqref{eq:feedbackcovariance} also satisfies~\eqref{eqn:Constraint on Sigma} with $H=-XK^*+ \frac{1}{2} B \, \Omega$. Conversely, if $X\succ 0$ satisfies~\eqref{eqn:Constraint on Sigma}, for $K = \frac{1}{2}\Omega \, B^*X^{-1} - H^*X^{-1}$, then $X$ also satisfies~\eqref{eq:feedbackcovariance} and $A-BK$ is Hurwitz.
Thus, the following statements are equivalent:
\bi
\item A matrix $X \succ 0$ qualifies as the stationary state covariance of~\eqref{eq:systemwithcontrol} via a suitable choice of state-feedback~\eqref{eq.v}.
	\vspace*{-.25cm}
\item A matrix $X \succ 0$ is a state covariance of~\eqref{eqn:State equation} for some stationary stochastic input $u$.
\ei

To clarify the connection between $K$ and the corresponding modeling filter for $u$, let
\begin{subequations}\label{eq:linearfeedback}
\be
\label{eq:linearoutput}
        u \;=\; -K \, x \,+\, w.
\ee
Substitution of~\eqref{eq.v} into~\eqref{eq:systemwithcontrol} yields
\begin{align}
	\label{eq.7b}
        \dot{x} 
        & ~ = ~
        ( A \,-\, BK ) \, x \;+\; B \, w 
        \\[0.cm]
        & ~ = ~
        A \, x \; + \; B \, u
        \non
\end{align}
\end{subequations}
which coincides with~\eqref{eqn:State equation}. Thus, $X$ can also be achieved by driving~\eqref{eqn:State equation} with $u$ given by~\eqref{eq:linearoutput}. The equivalence of~\eqref{eq:filtermodel} and~\eqref{eq:linearfeedback} is evident.
 {
Equation~\eqref{eq.7b} shows that a colored-in-time stochastic input process $u$ can be interpreted as a dynamical perturbation to system~\eqref{eqn:State equation}. This offers advantages from a computational standpoint, e.g., when conducting stochastic simulations; see Section~\ref{sec.example}.
}

In general, there is more than one choice of $K$ that yields a given feasible $X$. A criterion for the selection of an optimal feedback gain $K$,
can be to minimize 
   \[
        \lim_{t \, \rightarrow \, \infty}\bE \left( v^* (t) \, v(t) \right).
    \]
This optimality criterion relates to information theoretic notions of distance (Kullback-Leibler divergence) between corresponding models with and without control~\cite{dai1991stochastic, chen2014optimal, chegeopav14c}. Based on this criterion, the optimal feedback gain $K$ can be obtained by minimizing $\trace \, (K X K^*)$, subject to the linear constraint~\eqref{eq:feedbackcovariance}. This choice of $K$ characterizes an optimal filter of the form~\eqref{eq:linearfeedback}. {This filter is used in Section~\ref{sec.example} where we provide an illustrative example.}

	\vspace*{-2ex}
\section{Covariance completion and model complexity}
\label{sec.completion}

In Section \ref{sec.linearstochasticmodels}, we presented the structural constraints on the state covariance $X$ of an LTI system. We also proposed a method to construct a class of linear filters that generate the appropriate input process $u$ to account for the statistics in $X$. In many applications, the dynamical generator $A$ in~\eqref{eqn:State equation} is known. For example, in turbulent fluid flows the mean velocity can be obtained using numerical simulations of the Navier-Stokes equations and linearization around this equilibrium profile yields $A$ in~\eqref{eqn:State equation}. On the other hand, stochastic excitation often originates from disturbances that are difficult to model directly. To complicate matters, the state statistics may be only partially known {, i.e., only certain correlations between a limited number of states may be available. For example, such second-order statistics may reflect partial output correlations obtained in numerical simulations or experiments of the underlying physical system.}
Thus, we now introduce a framework for completing unknown elements of $X$ in a manner that is consistent with state-dynamics and, thereby, obtaining information about the spectral content and directionality of input disturbances to~\eqref{eqn:State equation}.


For colored-in-time {disturbance} $u$ that enters into the state equation {in {\em all directions},} through the identity matrix, condition~\eqref{eqn:rank Constraint on Sigma} is trivially satisfied. Indeed, any sample covariance $X$ can be generated by a linear model~\eqref{eqn:State equation} with $B = I$. {
In this case, the Lyapunov-like constraint~\eqref{eqn:Constraint on Sigma} simplifies to
\be
	\label{eq.lyap_B_identity}
	A \, X
	\; + \; 
	X \,A^*
	\;=\; 
	-H^*
	\; - \;
	H.
	\non
\ee
Clearly, this equation is satisfied with $H^* = -A \, X$.  With this choice of the cross-correlation matrix $H$, the dynamics represented by~\eqref{eq.7b} can be equivalently written as
\[
    \dot{x}
    \; = \;
    - \dfrac{1}{2}
    \,
    X^{-1}
    \, x
    \;+\; 
    w
\]
with a white disturbance $w$.
This demonstrates that {\em colored-in-time\/} forcing $u$ which {\em excites all degrees of freedom\/} can {\em completely overwrite the original dynamics}.
Thus, such an input disturbance} can trivially account for the observed statistics {{\em but}} provides no useful information about the underlying physics.


In our setting, the structure and size of the matrix $B$ in~\eqref{eqn:State equation} is not known {\em a priori\/}, which means that the direction of the input disturbances are not given. In most physical systems, disturbance can directly excite only a limited number of directions in the state space. For instance, in mechanical systems where inputs represent forces and states represent position and velocity, disturbances can only enter into the velocity equation. Hence, it is of interest to identify a disturbance model that involves a small number of input channels. This requirement can be formalized by restricting the input to enter into the state equation through a matrix $B \in \bbC^{n \times m}$ with $m < n$. Thus, our objective is to identify matrices $B$ and $H$ in~\eqref{eqn:Constraint on Sigma} to reproduce a partially known $X$ while striking an optimal balance with the complexity of the model; the complexity is reflected in the rank of $B$, i.e., the number of input channels. This notion of complexity is closely related to the signature of $Z$, which we discuss next.

	\vspace*{-2ex}
\subsection{The signature of $Z$}
\label{sec.signature}

As mentioned in Section \ref{sec.linearstochasticmodels}, the matrix $Z$ in~\eqref{eqn: Definition of Z} is not necessarily positive semidefinite. However, it is not arbitrary. We next examine admissible values of the {\em signature\/} on $Z$, i.e., the number of positive, negative, and zero eigenvalues. In particular, we show that the number of positive and negative eigenvalues of $Z$ impacts the number of input channels in the state equation~\eqref{eqn:State equation}.

There are two sets of constraints on $Z$ arising from~\eqref{eqn:Constraint left} and~\eqref{eqn:Constraint right}, respectively. The first one is a standard Lyapunov equation with Hurwitz $A$ and a given Hermitian $X\succ 0$. The second provides a link between the signature of $Z$ and the number of input channels in~\eqref{eqn:State equation}.

First, we study the constraint on the signature of $Z$ arising from~\eqref{eqn:Constraint left} which we repeat here,
 \begin{equation}
   \label{eqn:lyap}
   A \, X \; + \; X A^* \; = \; - Z.
   \end{equation}
The unique solution to this Lyapunov equation,
with Hurwitz $A$ and Hermitian $X$ and $Z$, is given by
   \begin{equation}
   \label{eqn:Lyapunov solution}
   X
   \; = \;
   \int_{0}^{\infty}
   \mre^{A t} \, Z \, \mre^{A^* t} \, \mrd t.
   \end{equation}
Lyapunov theory implies that if $Z$ is positive definite then $X$ is also positive definite. However, the converse is not true. Indeed, for a given $X \succ 0$, $Z$ obtained from~\eqref{eqn:lyap} is not necessarily positive definite. Clearly, $Z$ cannot be negative definite either, otherwise $X$ obtained from~\eqref{eqn:Lyapunov solution} would be negative semidefinite. We can thus conclude that~\eqref{eqn:lyap} does in fact introduce a constraint on the signature of $Z$. In what follows, the signature is defined as the triple
   \begin{equation}
   \non
   \In(Z) \;=\; \left(\pi(Z),\nu(Z),\delta(Z) \right)
   \end{equation}
where $\pi(Z), \nu(Z)$, and $\delta(Z)$ denote the number of positive, negative, and zero eigenvalues of $Z$, respectively.

Several authors have studied constraints on signatures of $A$, $X$, and $Z$ that are linked through a Lyapunov equation~\cite{tau61,ost62,dea95}. Typically, such studies focus on the relationship between the signature of $X$ and the eigenvalues of $A$ for a given $Z \succeq 0$. In contrast, \cite{sil07} considers the relationship between the signature of $Z$ and eigenvalues of $A$ for $X\succ 0$ and we make use of these results.

Let $\{\lambda_1, \ldots, \lambda_l\}$ denote the eigenvalues of $A$, $\mu_k$ denote the geometric multiplicity of $\lambda_k$, and
   \begin{equation}
   \label{eqn:geometric multiplicity}
   \mu(A)
   \; \DefinedAs \;
   \max_{1 \, \leq \, k \, \leq \, l} \, \mu_k.
   \non
   \end{equation}
The following result is a special case of~\cite[Theorem~2]{sil07}.

\begin{prop}
   \label{thm:proposition 1}
Let $A$ be Hurwitz and let $X$ be positive definite. For $Z=-(AX+X A^*)$,
   \begin{align}
   \pi(Z) & \, \ge \; \mu(A).
   \label{eqn:Constraint on positive}
   \end{align}
\end{prop}

{
To explain the nature of the constraint $\pi(Z)\ge \mu(A)$,
we first note that $\mu(A)$ is {\em the least number of input channels that are needed for system \eqref{eqn:State equation} to be controllable\/}~\cite[p.\ 188]{che95}.} Now consider the decomposition
   \[
   Z
   \; = \;
   Z_+ \; - \; Z_-
   \]
where $Z_+$, $Z_-$ are positive semidefinite matrices, and accordingly $X=X_+-X_-$ with $X_+$, $X_-$ denoting the solutions of the corresponding Lyapunov equations. Clearly, unless the above constraint \eqref{eqn:Constraint on positive} holds, $X_+$ cannot be positive definite. Hence, $X$ cannot be positive definite either.
Interestingly, there is no constraint on $\nu(Z)$ other than
\[
\pi(Z) \; + \; \nu(Z) \; \leq \; n
\]
which comes from the dimension of $Z$.

To study the constraint on the signature of $Z$ arising from \eqref{eqn:Constraint right}, we begin with a lemma, whose proof is provided in the appendix.
   \begin{lemma}
   \label{thm:lemma 0}
For a Hermitian matrix $Z$ decomposed as
   \begin{align*}
   Z \; = \; S \; + \; S^*
   \end{align*}
the following holds
   \begin{align*}
   \pi(Z) \; \leq \; \rank(S).
   \end{align*}
   \end{lemma}


Clearly, the same bound applies to $\nu(Z)$, that is,
   \begin{align*}
   \nu(Z) \; \leq \; \rank(S).
   \end{align*}

The importance of these bounds stems from our interest in decomposing $Z$ into summands of small rank. A decomposition of $Z$ into
$S+S^*$ allows us to identify input channels and power spectra by factoring
$S=BH^*$. The rank of $S$ coincides with the rank of $B$, that is, with the number of input channels in the state equation. Thus, it is of interest to determine the minimum rank of $S$ in such a decomposition and this is given in {Proposition~\ref{thm:lemma 1} (the proof is provided in the appendix).}

   \begin{prop}
   \label{thm:lemma 1}
For a Hermitian matrix $Z$ having signature $(\pi(Z),\nu(Z),\delta(Z))$,
   \begin{equation}
   \label{eqn:minrank}
   \min
   \left\{ \rank(S) | ~ Z \, = \, S \, + \, S^* \right\}
   \, = \,
   \max \left\{ \pi(Z),\nu(Z) \right\}.
   \non
   \end{equation}
   \end{prop}


We can now summarize the bounds on the number of positive and negative eigenvalues of the matrix $Z$ defined by~\eqref{eqn: Definition of Z}. By combining Proposition~\ref{thm:proposition 1} with Lemma~\ref{thm:lemma 0} we show that these upper bounds are dictated by the number of inputs in the state equation~\eqref{eqn:State equation}.

\begin{prop}\label{thm:theorem 1}
Let $X \succ 0$ denote the steady-state covariance of the state $x$ of a stable linear system \eqref{eqn:State equation} with $m$ inputs. If $Z$ satisfies the Lyapunov equation \eqref{eqn:lyap}, then
   \begin{subequations}
   \begin{align}
   0 & \, \le \; \nu(Z) \; \le \; m
   \non
   \\[0.1cm]
   \mu(A)  & \, \le \; \pi(Z) \; \le \; m.
   \non
   \end{align}
   \end{subequations}
\end{prop}

   \begin{proof}
From Section \ref{sec.linearstochasticmodels}, a state covariance $X$ satisfies
   \begin{equation}
   A \, X
   \; + \;
   X A^*
   \; = \;
   - BH^* \; - \; HB^*.
   \non
   \end{equation}
Setting $S=BH^*$,
   \begin{equation}
   Z
   \; = \;
   BH^* \; + \; HB^*
   \; = \;
   S \; + \; S^*.
   \non
   \end{equation}
From Lemma \ref{thm:lemma 0},
   \begin{equation}
   \max\{\pi(Z),\nu(Z)\}
   \; \leq \; \rank(S)
   \; \leq \;
   \rank(B)
   \; = \;
   m.
   \non
   \end{equation}
The lower bounds follow from Proposition \ref{thm:proposition 1}.
   \end{proof}

	\vspace*{-2ex}
\subsection{Decomposition of $Z$ into $B H^* + H B^*$}
\label{sec:fac}

Proposition~\ref{thm:lemma 1} expresses the possibility to decompose the matrix $Z$ into $BH^*+HB^*$ with $S=BH^*$ of minimum rank equal to $\max \left\{ \pi(Z),\nu(Z) \right\}$. Here, we present an algorithm that achieves this objective. Given $Z$ with signature $(\pi(Z),\nu(Z),\delta(Z))$, we can choose an invertible matrix $T$ to bring $Z$  into the following form\footnote{{The choice of $T$ represents a standard congruence transformation that brings $Z$ into canonical form.}}
   \begin{equation}
   \label{eq:FormZ}
   \hat{Z}
   \; \DefinedAs \;
   T\,Z\,T^*
   \; = \;
         2 \left[
         \begin{array}{ccc}
         I_{\pi} & 0 & 0\\
         0 & -I_{\nu} & 0\\
         0 & 0 & 0
         \end{array}
         \right]
   \end{equation}
where $I_{\pi}$ and $I_{\nu}$ are identity matrices of dimension $\pi(Z)$ and $\nu(Z)$~\cite[pages 218--223]{horjoh90}.
We first present a factorization of $Z$ for $\pi(Z) \leq \nu(Z)$. With
   \begin{equation}
       \hat{S}
   \; = \,
   \left[
   \begin{array}{cccc}
         I_{\pi} & -I_{\pi} & 0 & 0\\
         I_{\pi} & -I_{\pi} & 0 & 0\\
         0 & 0 & -I_{\nu - \pi} & 0\\
         0 & 0 & 0 & 0
   \end{array}
   \right]
   \label{eq:FormS}
   \end{equation}
we clearly have $\hat{Z} = \hat{S} + \hat{S}^*$. Furthermore, $\hat{S}$ can be written as $\hat{S} = \hat{B} \hat{H}^*$, where
   \begin{equation}
   \nonumber
   \hat{B}
   \; = \,
   \left[
   \begin{array}{cc}
   I_{\pi} & 0\\
   I_{\pi} & 0\\
   0 & I_{\nu - \pi}\\
   0 & 0
   \end{array}
   \right],
   ~~
   \hat{H}
   \; = \,
   \left[
   \begin{array}{cc}
   I_{\pi} & 0\\
   -I_{\pi} & 0\\
   0 & -I_{\nu - \pi}\\
   0 & 0
   \end{array}
   \right].
   \end{equation}
   In case $\nu(Z)=\pi(Z)$, $I_{\nu-\pi}$ and the corresponding row and column are empty.
Finally, the matrices $B$ and $H$ are determined by $B=T^{-1}\hat{B}$ and $H=T^{-1}\hat{H}$.

Similarly, for $\pi(Z) > \nu(Z)$, $Z$ can be decomposed into $BH^*+HB^*$ with $B=T^{-1}\hat{B}$,  $H=T^{-1}\hat{H}$, and
   \begin{equation}
   \non
   \hat{B}
   \; = \,
   \left[
   \begin{array}{cc}
   I_{\pi-\nu} & 0\\
   0 & I_\nu\\
   0 & I_\nu\\
   0 & 0
   \end{array}
   \right],
   ~~
   \hat{H}
   \; = \,
   \left[
   \begin{array}{cc}
   I_{\pi-\nu} & 0\\
   0 & I_\nu\\
   0 & -I_\nu\\
   0 & 0
   \end{array}
   \right].
   \end{equation}
Note that both $B$ and $H$ are full column-rank matrices.

	\vspace*{-2ex}
\subsection{Covariance completion problem}
\label{sec.ccp}

Given the dynamical generator $A$ and partially observed state correlations, we want to obtain a low-complexity model for the disturbance that can explain the observed entries of $X$. Here the complexity is reflected by the number of input channels, i.e., the rank of the input matrix $B$. {Clearly, $\rank(B) \geq \rank(S)$. Furthermore, any $S$ can be factored as $S = B H^*$ with $\rank(B) = \rank(S)$ via, e.g., singular value decomposition. Thus, we focus on minimizing the rank of $S$.
}

{Rank minimization is a difficult problem because ${\rm rank}(\cdot)$ is a non-convex function.} Recent advances have demonstrated that the minimization of the nuclear-norm (i.e., the sum of the singular values)
    \[
        \| S \|_*
        \; \DefinedAs \;
        \sum_{i \, = \, 1}^n \sigma_i (S)
    \]
represents a good proxy for rank minimization~\cite{fazhinboy01,faz02,canrec09,kesmonoh10,recfazpar10,canpla10,cantao10,chasanparwil11}. We thus formulate the following  matrix completion problem:
	\bi
\item[] {\em Given a Hurwitz $A$ and the matrix $G$, determine matrices $X = X^*$ and $Z = S + S^*$ from the solution to}
\begin{align}
	\ba{cl}
	\minimize\limits_{S, \, X}
	&
	\| S \|_*
	\\[.15cm]
	\subject &
	A X \,+\, X A^* \,+\, S \,+\, S^* \; = \; 0
	\\[.1cm]
	&
	(C \, X \, C^*) \circ E \,-\, G \; = \; 0
	\\[.1cm]
	& X \,\succeq\, 0.
	 \ea
	\label{LR}
\end{align}
	\ei
	
In the above, the matrices $A$, $C$, $E$, and $G$ represent problem data, while $S$, $X \in \bbC^{n\times n}$ are optimization variables. The entries of  {the Hermitian matrix} $G$ represent partially known second-order statistics which reflect output correlations provided by numerical simulations or experiments of the underlying physical system. The symbol $\circ$ denotes elementwise matrix multiplication and the matrix $E$ is the structural identity defined by
	\be
	E_{ij} \,=\,
	\left\{
	\ba{ll}
	1,
	&~
	G_{ij} ~ \text{is available}
	\\[.1cm]
	0,
	&~
	G_{ij}~ \text{is unavailable.}
	\ea
	\right.
	\non
\ee

The constraint set in~(\ref{LR}) represents the intersection of the positive semidefinite cone and two linear subspaces. These are specified by the Lyapunov-like constraint, which is imposed by the linear dynamics, and the linear constraint which relates $X$ to the available entries of the steady-state output covariance matrix
\be
    \displaystyle{\lim_{t \, \rightarrow \, \infty}}
    \bE
    \left( y (t)  \, y^*(t) \right)
    \; = \;
    C\, X \, C^*.
    \non
\ee

	
As shown in Proposition~\ref{thm:lemma 1}, {\em minimizing the rank of $S$ is equivalent to minimizing\/} $\max \, \{\pi(Z),\nu(Z)\}$. Given $Z$, there exist matrices $Z_+\succeq 0$ and $Z_-\succeq 0$ with $Z=Z_+-Z_-$ such that $\rank(Z_+)=\pi(Z)$ and $\rank(Z_-)=\nu(Z)$. Furthemore, any such decomposition of $Z$ satisfies $\rank(Z_+)\geq \pi(Z)$ and $\rank(Z_-)\geq \nu(Z)$. Thus, instead of~\eqref{LR}, we can alternatively consider the following convex optimization problem, which aims at minimizing $\max \left\{\pi(Z),\nu(Z) \right\}$,
\begin{align}
	\ba{cl}
	\minimize\limits_{X,\, Z_+,\, Z_-}
	&
	\max \left\{ \trace(Z_+), \, \trace(Z_-) \right\}
	\\[.15cm]
	\subject &
	A \, X \,+\, X A^* \,+\, Z_+ \,-\, Z_- \; = \; 0
	\\[.1cm]
	&
	(C \, X \, C^*) \circ E \,-\, G \; = \; 0
	\\[.1cm]
	& X \,\succeq\, 0,~ Z_+ \,\succeq\, 0,~ Z_- \,\succeq\, 0.
	 \ea
	\label{LR1}
\end{align}
Both~\eqref{LR} and~\eqref{LR1} can be solved efficiently using standard SDP solvers~\cite{cvx,boyvan04} for small- and medium-size problems. 
{Note that~\eqref{LR} and~\eqref{LR1} are obtained by relaxing the rank function to the nuclear norm and the signature to the trace, respectively. Thus, even though the original non-convex optimization problems are equivalent to each other, the resulting convex relaxations~\eqref{LR} and~\eqref{LR1} are not, in general.}

In Section \ref{sec.algorithms}, we develop an efficient customized algorithm which solves the following {\em covariance completion\/} problem
\begin{align}
	\ba{cl}
	\minimize\limits_{X,\, Z}
	&
	-\logdet X \,+\, \gamma\,\norm{Z}_*
	\\[.15cm]
	\subject
	&
	~A \, X \,+\, X A^* \,+\, Z  \,=\, 0
	\\[.15cm]
	&
	\,\,\left(C X C^* \right)\circ E \,-\, G \,=\, 0.
	 \ea
	 \tag{CC}
	\label{eq.CCP}
\end{align}
 {For any $Z$ there exists a decomposition $Z = Z_+ - Z_-$ with $Z_+$, $Z_-\succeq 0$ such that
\[
        \norm{Z}_*
        \;=\;
        \trace(Z_+) \,+\, \trace(Z_-).
\]
Since
\[  
	 \trace(Z_+) \,+\, \trace(Z_-)   
        \; \ge \;
        \max \{ \trace(Z_+), \, \trace(Z_-) \},
\]
the solution to~\eqref{eq.CCP} provides a possibly suboptimal solution to~\eqref{LR1}.}
In recent work~\cite{linjovgeoCDC13, zarjovgeoACC14}, we considered \eqref{eq.CCP}
in the absence of the logarithmic barrier function. However, in that work, the corresponding semidefinite $X$ is not suitable for synthesizing the input filter~\eqref{eq:filtermodel} because $X^{-1}$ appears in the expression for $K$; cf.~\eqref{eq.K}. Furthermore, as we show in Section~\ref{sec.algorithms}, another benefit of using the logarithmic barrier is that it ensures strong convexity of the smooth part of the objective function in~\eqref{eq.CCP} which is exploited in our customized algorithm.

	\vspace*{-2ex}
\section{Customized algorithm for solving the covariance completion problem}
\label{sec.algorithms}

We begin this section by bringing~\eqref{eq.CCP} into a form  {that} is convenient for alternating direction methods. We then study the optimality conditions, formulate the dual problem, and develop a customized Alternating Minimization Algorithm (AMA) for~\eqref{eq.CCP}. Our customized algorithm allows us to exploit the respective structure of the logarithmic barrier function and the nuclear norm, thereby leading to an efficient implementation that is well-suited for large systems.

We note that AMA was originally developed by Tseng~\cite{tse91} and its enhanced variants have been recently presented in~\cite{golodosetbar14, dalraj14} and used, in particular, for estimation of sparse Gaussian graphical models. 
{In Section~\ref{sec.AMApg}, we show that AMA can be equivalently interpreted as a proximal gradient algorithm on the dual problem. This allows us to establish theoretical results regarding the convergence of AMA when applied to the optimization problem~\eqref{eq.CCP}. It also enables a principled step-size selection aimed at achieving sufficient dual ascent.
}

In~\eqref{eq.CCP}, $\gamma$ determines the importance of the nuclear norm relative to the logarithmic barrier function.
The convexity of~(\ref{eq.CCP}) follows from the convexity of the objective function
\[
	J_p(X,Z) \;\DefinedAs\; -\logdet X \,+\, \gamma\,\norm{Z}_*
\]	
and the convexity of the constraint set. Problem~(\ref{eq.CCP}) can be equivalently expressed as follows,
	\begin{align} 
	\ba{cl}
	\minimize\limits_{X,\, Z}
	&
	-\logdet X \,+\, \gamma\,\norm{Z}_*
	\\[.1cm]
	\subject
	&
	~\cA\, X \,+\, \cB\, Z \,-\, \cC  \,=\, 0,
	 \ea
	 \tag{CC-1}
	\label{eq.CCP1}
\end{align} 
where the constraints are now given by
\be
	\tbo{\cA_1}{\cA_2} X \,+\, \tbo{I}{0} Z \,-\, \tbo{0}{G} \,=\, 0.
	\non
\ee
Here, $\cA_1: \bbC^{n\times n} \to \bbC^{n\times n}$ and $\cA_2: \bbC^{n\times n} \to \bbC^{p \times p}$ are linear operators, with
\be
	\ba{c}
	\cA_1(X) \; \DefinedAs \; A \, X \; + \; XA^*
	\\[.15cm]
	\cA_2(X) \; \DefinedAs \; (C \, X \, C^*)\circ E
	\ea
	\non
\ee
{and their adjoints, with respect to the standard inner product $\left<M_1, M_2\right> \DefinedAs \trace \, (M_1^* M_2)$, are given by
\be
	\ba{c}
	\cA_1^\dagger (Y) \; = \; A^* \, Y \;+\; Y A
	\\[.15cm]
	\cA_2^\dagger (Y) \; = \; C^* (E\circ Y) \, C.
	\ea
	\non
\ee}


	\vspace*{-2ex}
{
\subsection{SDP formulation and the dual problem}
	\label{sec.optimality_dual}
}

{
By splitting $Z$ into positive and negative definite parts,
\be
	Z \; = \; Z_+ \; - \; Z_-,~~~~Z_+ \; \succeq \; 0,~~~~Z_- \; \succeq \; 0
	\non
\ee
it can be shown~\cite[Section 5.1.2]{faz02} that~(\ref{eq.CCP1}) can be cast as an SDP,
\begin{align} 
	\ba{cl}
	\minimize\limits_{X,\, Z_+,\, Z_-}
	&
	-\logdet X
	\; + \; 
	\gamma
	\left( \trace\, (Z_+) \, + \, \trace\, (Z_-)\right)
	\\[.15cm]
	\subject
	& \cA_1(X) \,+\, Z_+ \, - \, Z_- \,=\, 0
	\\[0.15cm]
	& \cA_2 (X) \,-\, G \,=\, 0
	\\[0.15cm]
	& Z_+ \, \succeq \, 0,~~~~Z_- \, \succeq \, 0.
	 \ea
	 \tag{P}
	 \label{eq.CCPsdp}
\end{align}
We next use this SDP formulation to derive the Lagrange dual of the covariance completion problem~\eqref{eq.CCP1}. 
}

{
\begin{prop}
   \label{prop:dual}
The Lagrange dual of~(\ref{eq.CCPsdp}) is given by
\begin{equation}
	\ba{cl}
	\maximize\limits_{Y_1,\, Y_2}
	&
	\logdet\left( \cA_1^\dagger(Y_1) 
	\, + \, 
	\cA_2^\dagger(Y_2) \right) 
	\, - \; 
	\left<G,\, Y_2 \right> 
	\, + \; n
	\\[.25cm]
	\subject
	&
	\norm{Y_1}_2 
	\, \leq \, 
	\gamma
	 \ea
	 \tag{D}
	 \label{eq.dual}
\end{equation}
\end{prop}
where Hermitian matrices $Y_1$, $Y_2$ are the dual variables associated with the equality constraints in~\eqref{eq.CCPsdp}.}

{
\begin{proof}
	The Lagrangian of~\eqref{eq.CCPsdp} is given by
\be
	\ba{l}
	\cl(X,Z_\pm; Y_1, Y_2, \Lambda_\pm) 
	\; = \;
	-\logdet X
	\, +\,
	\gamma\, \trace \left(Z_+ \,+\, Z_-\right) \,-\, \left<\Lambda_+,\, Z_+\right> \,-\, \left<\Lambda_-,\, Z_-\right>
	\; +
	\\[.15cm]
	\hfill
	\left<Y_1,\, \cA_1(X) + Z_+ - Z_-\right> \,+\, \left<Y_2,\, \cA_2(X) - G \right>
	\ea
	\label{eq.lagrangian}
\ee
where Hermitian matrices $Y_1$, $Y_2$, and $\Lambda_\pm \succeq 0$ are Lagrange multipliers associated with the equality and inequality constraints in~\eqref{eq.CCPsdp}. Minimizing the Lagrangian with respect to $Z_+$ and $Z_-$ yields
\be
	\ba{l}
	\gamma\, I \; - \; \Lambda_+ \;+\; Y_1 \; \succeq \; 0,~~~~ Z_+ \; \succeq \; 0
	\\[.15cm]
	\gamma\,I \; - \; \Lambda_- \; - \; Y_1 \; \succeq \; 0,~~~~ Z_- \; \succeq \; 0.
	\ea
	\non
\ee
Because of the positive semi-definiteness of the dual variables $\Lambda_+$ and $\Lambda_-$, we also have that
\be
	\ba{rcccl}
	Y_1 \; + \; \gamma\,I &\!\!\succeq\!\!\!& \Lambda_+ &\!\!\!\succeq\!\!& 0
	\\[.15cm]
	- Y_1 \;+\; \gamma\,I &\!\!\succeq\!\!\!& \Lambda_- &\!\!\!\succeq\!\!& 0,
	\ea
	\non
\ee
which yields the constraint in~\eqref{eq.dual}
\be
	-\gamma \, I \; \preceq \; Y_1 \; \preceq \; \gamma \, I ~~\Longleftrightarrow~~ \norm{Y_1}_2 \; \leq \; \gamma.
	\label{eq.dual_constraint}
\ee
On the other hand, minimization of $\cL$ with respect to $X$ yields
\be
	X^{-1} \; = \; \cA_1^\dagger (Y_1) \; + \; \cA_2^\dagger (Y_2) \; \succ \; 0.
	\label{eq.Xoptimality}
\ee
Substitution of~\eqref{eq.Xoptimality} into~\eqref{eq.lagrangian} in conjunction with complementary slackness conditions
\be
	\ba{rcl}
	\left< \gamma\,I \,-\, \Lambda_+ \,+\, Y_1,\, Z_+ \right>
	&\!\!=\!\!&
	0
	\\[.15cm]
	\left< \gamma\,I \,-\, \Lambda_- \,-\, Y_1,\, Z_- \right>
	&\!\!=\!\!&
	0
	\ea
	\non
\ee
can be used to obtain the Lagrange dual function
	\be
	\ba{rcl}
	J_d(Y_1, Y_2)
	& \!\! = \!\! &
	\inf\limits_{X,\, Z_+,\, Z_-}
	\cl(X,Z_\pm; Y_1, Y_2, \Lambda_\pm) 
	\\[0.15cm]
	& \!\! = \!\! &
	\logdet \left( 
	\cA_1^\dagger(Y_1) 
	\; + \;
	\cA_2^\dagger(Y_2) 
	\right) 
	 \; - \; 
	\left<G,\, Y_2 \right> 
	\; + \; 
	n.
	\ea
	\non
\ee
\end{proof}
}

The dual problem~(\ref{eq.dual}) is a convex optimization problem with variables $Y_1 \in \bbC^{n\times n}$ and $Y_2 \in \bbC^{p \times p}$. These variables are dual feasible if the constraint in~(\ref{eq.dual}) is satisfied. In the case of primal and dual feasibility, any dual feasible pair $(Y_1,\, Y_2)$ gives a lower bound on the optimal value $J_p^\star$ of the primal problem~(\ref{eq.CCPsdp}). As we show next, the alternating minimization algorithm of Section~\ref{sec.AMA} can be interpreted as a proximal gradient algorithm on the dual problem and is developed to achieve sufficient dual ascent and satisfy~(\ref{eq.Xoptimality}).


	\vspace*{-2ex}
\subsection{Alternating Minimization Algorithm (AMA)}
\label{sec.AMA}

The logarithmic barrier function in~(\ref{eq.CCP}) is strongly convex over any compact subset of the positive definite cone~\cite{banelgdas08}. This makes it well-suited for the application of AMA, which requires strong convexity of the smooth part of the objective function~\cite{tse91}.

The augmented Lagrangian associated with~(\ref{eq.CCP1}) is
	\be		
	\ba{l}
	\cl_\rho (X, Z; Y_1, Y_2) 
	\; = \;
	\ds{-\logdet X \, + \, \gamma\, \norm{Z}_*}  
	\, +\,
	
	{
	\left<Y_1,\, \cA_1 (X) + Z \right>
	} 
	\, + \,
	
	{
	\left<Y_2,\, \cA_2 (X) - G \right>
	}
	~ +
	\\[0.15cm] 
	\hfill
	{ 
	\ds{
	\frac{\rho}{2}\, \norm{\cA_1 (X) \, + \, Z}_F^2}
	\; + \;
	\ds{
	\frac{\rho}{2}\, \norm{\cA_2 (X) \, - \, G}_F^2}
	}
	\ea
	\non
	\ee
{where $\rho$ is a positive scalar and $\norm{\cdot}_F$ is the Frobenius norm.}

AMA {consists of the following steps:}
{
\begin{subequations}
	\label{eq.AMA_steps}
	\begin{eqnarray}
	\hspace{-.5cm} X^{k+1} &\!\!\DefinedAs\!\!& \argmin\limits_{X} \,  \cl_0\, ( X,\, Z^k; \, Y_1^k,\, Y_2^k)
	\label{eq.AMA_Xmin}
	\\
	\hspace{-.5cm} Z^{k+1} &\!\!\DefinedAs\!\!& \argmin\limits_{Z} \, \cl_{\rho}\, ( X^{k+1},\, Z; \, Y_1^k,\, Y_2^k)
	\label{eq.AMA_Zmin}
	\\
	\hspace{-.5cm} 
	Y_1^{k+1}
	&\!\!\DefinedAs\!\!& 
	Y_1^k \,+\, \ds{\rho \left( \cA_1 (X^{k+1}) \, + \, Z^{k+1} \right)}
	\label{eq.AMA_dualupdate_1}
	\\
	\hspace{-.5cm} 
	Y_2^{k+1} 
	&\!\!\DefinedAs\!\!& 
	Y_2^k \,+\, \ds{\rho \left( \cA_2 (X^{k+1}) \, - \, G \right)}.
	\label{eq.AMA_dualupdate_2}
	\end{eqnarray}
\end{subequations}
}
These terminate when the duality gap
\[
	\Delta_{\mathrm{gap}}
	\; \DefinedAs \;
	-\logdet X^{k+1} \,+\, \gamma\,\norm{Z^{k+1}}_* \, - \, J_d \left(Y_1^{k+1}, Y_2^{k+1} \right)
\]
and the primal residual
\[
	\Delta_{\mathrm{p}} \;\DefinedAs\; \norm{\cA \, X^{k+1} \,+\, \cB\, Z^{k+1} \,-\, \cC}_F
\]
are sufficiently small, i.e., $|\Delta_{\mathrm{gap}}| \leq {\eps_1}$ and $\Delta_{\mathrm{p}} \leq {\eps_2}$. 
In the $X$-minimization step~\eqref{eq.AMA_Xmin}, AMA minimizes the Lagrangian $\cl_0$ {with respect to $X$}. This step is followed by a $Z$-minimization step~\eqref{eq.AMA_Zmin} in which the augmented Lagrangian $\cl_{\rho}$ is minimized with respect to $Z$. Finally, the Lagrange multipliers, $Y_1$ and $Y_2$, are updated based on the primal residuals with the step-size $\rho$.

In contrast to the Alternating Direction Method of Multipliers~\cite{boyparchupeleck11}, which minimizes the augmented Lagrangian $\cl_{\rho}$ in both $X$- and $Z$-minimization steps, AMA updates $X$ via minimization of the standard Lagrangian $\cl_0$. As shown below, in~\eqref{eq.Xsol}, use of AMA leads to a {\em closed-form expression\/} for $X^{k+1}$. Another differentiating aspect of AMA is that it works as a proximal gradient on the dual function; {see Section~\ref{sec.AMApg}}. This allows us to select the step-size $\rho$ in order to achieve sufficient dual ascent.

\subsubsection{Solution to the X-minimization problem~\eqref{eq.AMA_Xmin}}

At the $k$th iteration of AMA, minimizing the Lagrangian $\cl_0$ with respect to $X$ for fixed {$\{Z^{k},\, Y_1^k,\, Y_2^k\}$ yields
\begin{align}
	X^{k+1} 
	\; = \, \left(\cA^\dagger \left(Y_1^{k}, Y_2^{k} \right) \right)^{-1}
	\; = \,
	\left(
	\cA_1^\dagger (Y_1^k)
	\,+\,
	\cA_2^\dagger (Y_2^k)
	\right)^{-1}.
	\label{eq.Xsol}
\end{align}
}

\subsubsection{Solution to the Z-minimization problem~\eqref{eq.AMA_Zmin}}
\label{sec.AMA_Zmin}

For fixed $\{X^{k+1},\, Y_1^k,\, Y_2^k\}$, the augmented Lagrangian $\cl_{\rho}$ is minimized with respect to $Z$,
\begin{align}
	\ba{cl}
	\minimize\limits_{Z}  & \ds{\gamma\, \norm{Z}_* \; + \; \frac{\rho}{2} \, \norm{Z \, - \, V^k}_F^2}.
	\ea
	\label{eq.Zmin_step}
\end{align}
{By computing the singular value decomposition of the symmetric matrix
	\[
	V^k
	\; \DefinedAs \;
	-\left(\cA_1 ( X^{k+1} ) \, + \, (1/\rho)Y_1^k\right)
	\; = \; 
	U\,\Sigma\, U^*,
	\]
where $\Sigma$ is the diagonal matrix of the singular values $\sigma_i$ of $V^k$, the solution to~(\ref{eq.Zmin_step}) is obtained by singular value thresholding~\cite{caicanshe10},}
\[
	Z^{k+1} \;=\; \cS_{\gamma/\rho} (V^k).
\]
The soft-thresholding operator $\cS_\tau$ is defined as 
\[
	\cS_\tau (V^k) \; \DefinedAs\;U \, \cS_\tau (\Sigma) \,U^*,\,
	~~ 
	\cS_\tau (\Sigma) \;=\; \diag \left( \left(\sigma_i \,-\, \tau \right)_+ \right)
\]
with $a_+ \DefinedAs \max \, \{a, 0\}$. 


\subsubsection{Lagrange multiplier update}

The expressions for $X^{k+1}$ and $Z^{k+1}$ can be used to bring~\eqref{eq.AMA_dualupdate_1} and~\eqref{eq.AMA_dualupdate_2} into the following form
\[
	\ba{rcl}
	Y_1^{k+1} 
	&\!\!=\!\!& 
	\cT_{\gamma} \left(Y_1^k \,+\, \rho\, \cA_1 (X^{k+1}) \right)
	\\[.2cm]
	Y_2^{k+1} 
	&\!\!=\!\!& 
	Y_2^k \,+\, \rho \left( \cA_2 ( X^{k+1} ) \, - \, G \right).
	\ea
\]
For Hermitian matrix $M$ with singular value decomposition $M = U\,\Sigma\, U^*$, $\cT_{\tau}$ is the saturation operator, 
\[
	\ba{rrl}
	\cT_\tau (M) 
	& \!\! \DefinedAs \!\! &
	U\, \cT_\tau (\Sigma)\, U^*
	\\[0.15cm]
	\cT_\tau (\Sigma) 
	& \!\! = \!\! &
	\diag \left( \min \left( \max( \sigma_i, -\, \tau ), \tau \right) \right)
	\ea
\]
which restricts the singular values of $M$ between $-\tau$ and $\tau$. 
The saturation and soft-thresholding operators are related via
\be
	M \;=\; \cT_\tau(M) \,+\, \cS_\tau(M).
	\label{eq.clip_identity}
\ee
{
The above updates of Lagrange multipliers guarantee dual feasibility at each iteration, i.e., $\norm{Y_1^{k+1}}_2 \leq \gamma$ for all $k$,} which justifies the choice of stopping criteria in ensuring primal feasibility of the solution.


\subsubsection{Choice of step-size for the dual update~{\eqref{eq.AMA_dualupdate_1},~\eqref{eq.AMA_dualupdate_2}}}
\label{sec.stepsize_rule}

We follow an enhanced variant of AMA~\cite{dalraj14} which utilizes an adaptive BB step-size selection~\cite{barbor88} in~\eqref{eq.AMA_Zmin},~{\eqref{eq.AMA_dualupdate_1}, and~\eqref{eq.AMA_dualupdate_2}} to guarantee sufficient dual ascent and positive definiteness of $X$. Our numerical experiments indicate that this {heuristic} provides substantial acceleration relative to the use of a fixed step-size. Since the standard BB step-size may not always satisfy the feasibility or the sufficient ascent conditions, we employ backtracking to determine an appropriate step-size.

At the $k$th iteration of AMA, an initial step-size,
	{
\be
\label{eq.BBstepsize}
	\ba{l}
	\rho_{k,0}
	\; = \,
	\dfrac{\ds{\sum_{i \, = \, 1}^2} \, \norm{Y_i^{k+1} - Y_i^{k}}_F^2}
	{\ds{\sum_{i \, = \, 1}^2} \, \left<Y_i^{k+1} - Y_i^{k},\, \nabla J_d(Y_i^{k}) - \nabla J_d(Y_i^{k+1}) \right>},
	\ea
\ee
}
is adjusted through a backtracking procedure to guarantee positive definiteness of the subsequent iterate of~(\ref{eq.AMA_Xmin}) and sufficient ascent of the dual function,
{
\begin{subequations}
	\label{eq.AMA_conditions}
	\begin{eqnarray}
	\cA^\dagger \left(Y_1^{k+1},\, Y_2^{k+1}\right) 
	&\!\! \succ \!\!&
	0
	\label{eq.AMA_Xpos}
	\\
	J_d \left(Y_1^{k+1},\, Y_2^{k+1}\right) 
	&\!\! \ge \!\!&
	J_d \left(Y^k\right) 
	\;+\;
	\ds{\sum_{i \, = \, 1}^2} 
	\left(
	\left<\nabla J_d(Y_i^k),\, Y_i^{k+1} - Y_i^k \right>
	+
	\dfrac{1}{2\rho_k} \, \norm{Y_i^{k+1} - Y_i^k}_F^2
	\right).
	\label{eq.AMA_suffascent}
	\end{eqnarray}
\end{subequations}
}
Here, $\nabla J_d$ is the gradient of the dual function {and the right-hand-side of~(\ref{eq.AMA_suffascent}) is a local quadratic approximation of the dual objective around $Y_1^k$ and $Y_2^k$. Furthermore,}~(\ref{eq.AMA_Xpos}) guarantees the positive definiteness of $X^{k+1}$; cf.~\eqref{eq.Xsol}. 

Our customized AMA is summarized as Algorithm~\ref{alg.AMA}.

\begin{algorithm}
\caption{Customized Alternating Minimization Algorithm}
\label{alg.AMA}
\begin{algorithmic}
\STATE \textbf{input:} $A$, $G$, $\gamma > 0$, tolerances {$\eps_1$, $\eps_2$}, and backtracking constant $\beta\in (0, 1)$.
\vspace*{0.15cm}
\STATE \textbf{initialize:} $k=0$, $\rho_{0,0}=1$, $\Delta_{\mathrm{gap}} = \Delta_{\mathrm{p}} = 2\eps_1$, $Y_2^0 = \mathrm{O}_{n\times n}$, and choose $Y_1^0$ such that $\cA_1^\dagger(Y_1^0)=(\gamma/\norm{Y_1^0}_2) I_{n\times n}$.
\vspace*{0.05cm}
\STATE \textbf{while:} $| \Delta_{\mathrm{gap}} | > {\eps_1}$ {and} $\Delta_{\mathrm{p}} > {\eps_2}$,
\\[.3cm]
\begin{tabular}{rcl}
\quad $X^{k+1}$ &$\!\!\!\!=\!\!\!\!$& $(\cA^\dagger (Y_1^k, Y_2^k))^{-1}$
\end{tabular}
\\[.2cm]
~\,\quad compute $\rho_k$: Largest feasible step in $\{\beta^j \rho_{k,0}\} _{j=0,1,\ldots}$
\\[.15cm]
~\,\quad such that $Y_1^{k+1}$ and $Y_2^{k+1}$ satisfy~\eqref{eq.AMA_conditions}
\\[.3cm]
\begin{tabular}{rcl}
\quad $Z^{k+1}$ &$\!\!\!\!=\!\!\!\!$& $\argmin\limits_{Z} \, \cl_{\rho_k}\, ( X^{k+1},\, Z,\, Y_1^k,\, Y_2^k)$
\\[.35cm]
$Y_1^{k+1}$ &$\!\!\!\!=\!\!\!\!$& $Y_1^k \,+\, \ds{\rho \left( \cA_1 (X^{k+1}) \, + \, Z^{k+1} \right)}$
\\[.25cm]
$Y_2^{k+1}$ &$\!\!\!\!=\!\!\!\!$& $Y_2^k \,+\, \ds{\rho \left( \cA_2 (X^{k+1}) \, - \, G \right)}$
\\[.25cm]
$\Delta_{\mathrm{p}}$ &$\!\!\!\!=\!\!\!\!$& $\norm{\cA\, X^{k+1} \,+\, \cB\, Z^{k+1} \,-\, \cC}_F$
\\[.25cm]
{$\Delta_{\mathrm{gap}}$} &{$\!\!\!\!=\!\!\!\!$}& {$-\logdet X^{k+1} \,+\, \gamma\,\norm{Z^{k+1}}_* \, -\, J_d \left(Y_1^{k+1}, Y_2^{k+1} \right)$}
\\[.25cm]
$k$ &$\!\!\!\!=\!\!\!\!$& $k+1$
\\[.2cm]
\end{tabular}
\\
{~\,\quad choose $\rho_{k,0}$ based on~\eqref{eq.BBstepsize}} 
\vspace*{0.15cm}
\STATE \textbf{endwhile}
\vspace*{0.15cm}
\STATE \textbf{output:} $\eps$-optimal solutions, $X^{k+1}$ and $Z^{k+1}$.
\end{algorithmic}
\end{algorithm}

\subsubsection{Computational complexity}
	\label{sec.complexity}

The $X$-minimization step in AMA involves a matrix inversion, which takes $O(n^3)$ operations. Similarly, the $Z$-minimization step amounts to a singular value decomposition and it requires $O(n^3)$ operations.  {Since this step is embedded within an iterative backtracking procedure for selecting the step-size $\rho_k$ (cf. Section~\ref{sec.stepsize_rule}), if the step-size selection takes $q$ inner iterations the total computational cost for a single iteration of AMA is $O(q n^3)$.} In contrast, the worst-case complexity of standard SDP solvers is $O(n^6)$.

\vspace*{-2ex}

\subsection{Comparison with ADMM}

Another splitting method that can be used to solve the optimization problem~\eqref{eq.CCP} is the Alternating Direction Method of Multipliers (ADMM). This method is well-suited for large-scale and distributed optimization problems and it has been effectively employed in low-rank matrix recovery~\cite{taoyua11}, sparse covariance selection~\cite{yua12}, image denoising and magnetic resonance imaging~\cite{golosh09}, sparse feedback synthesis~\cite{linfarjovACC12,linfarjovTAC13admm}, system identification~\cite{ayaszn12,hanliuvan12,liuhanvan13}, and many other applications~\cite{boyparchupeleck11}. In contrast to AMA, ADMM minimizes the augmented Lagrangian in each step of the iterative procedure. In addition, ADMM does not have efficient step-size selection rules. Typically, either a constant step-size is selected or the step-size is adjusted to keep the norms of primal and dual residuals within a constant factor of one another~\cite{boyparchupeleck11}. 

While the $Z$-minimization step is equivalent to that of AMA, the $X$-update in ADMM is obtained by minimizing the augmented Lagrangian. This amounts to solving the following optimization problem
\be
	\ba{cl}
	\minimize\limits_{X}
	&\!\!
	-\logdet X
	\;+\;
	\dfrac{\rho}{2} \, \ds{\sum_{i \, = \, 1}^2} \,\norm{\cA_i (X) \,-\, U_i^k}_F^2 	
	 \ea
	 \label{eq.admmXmin}
\ee
where $U_1^k \DefinedAs -\left( Z^k + (1/\rho)Y_1^k \right)$ and $U_2^k \DefinedAs G - (1/\rho)Y_2^k$. From first order optimality conditions we have
\[
	-\,X^{-1} 
	\,+\, 
	\rho\, \cA_1^\dagger ( \cA_1 (X) - U_1^k )
	\,+\, 
	\rho\, \cA_2^\dagger ( \cA_2 (X) - U_2^k )
	\;=\;
	0.
\]
Since $\cA_1^\dagger \cA_1$ and $\cA_2^\dagger \cA_2$ are not unitary operators, the $X$-minimization step {\em does not have an explicit solution}.

In what follows, we use a proximal gradient method~\cite{parboy13} to update $X$. By linearizing the quadratic term in~\eqref{eq.admmXmin} around the current inner iterate $X_i$ and adding a quadratic penalty on the difference between $X$ and $X_i$, $X_{i+1}$ is obtained as the minimizer of 
\be
	-\logdet X
	\; + \;
	\rho\, \ds{\sum_{j \, = \, 1}^2} \inner{\cA_j^\dagger\left(\cA_j (X_i) - U_j^k \right)}{X}
	\; + \;
	\dfrac{\mu}{2} \, \norm{X \,-\, X_i}_F^2.
	\label{eq.proximalXstep}
\ee
To ensure convergence of the proximal gradient method~\cite{parboy13}, the parameter $\mu$ has to satisfy $\mu \ge \rho\, \lambda_{\max}( \cA_1^\dagger \cA_1 + \cA_2^\dagger \cA_2)$, where we use power iteration to compute the largest eigenvalue of the operator $ \cA_1^\dagger \cA_1 + \cA_2^\dagger \cA_2$.

By taking the variation of~\eqref{eq.proximalXstep} with respect to $X$, we obtain the first order optimality condition
\be
	\label{eq.x-min-sol}
	\ba{rcl}
	\mu\, X \,-\, X^{-1}
	&\!\! = \!\!&
	\mu\, X_i  
	\,-\,
	\rho\, 
	\ds{\sum_{j \, = \, 1}^2} \cA_j^\dagger\left(\cA_j (X_i) - U_j^k \right).
	\ea
\ee
The solution to~\eqref{eq.x-min-sol} is given by  
	\[
	X_{i+1} 
	\; = \; 
	V \, \diag \left( g \right) V^*,
	\]
where the $j$th entry of the vector $g \in \bbR^{n}$ is given by
\[
	\ds{g_j
	\;=\;
	\dfrac{\lambda_j}{2\mu} \,+\, \sqrt{\left(\dfrac{\lambda_j}{2\mu}\right)^2 \,+\, \dfrac{1}{\mu}}}.
\]
Here, $\lambda_j$'s are the eigenvalues of the matrix on the right-hand-side of~\eqref{eq.x-min-sol} and $V$ is the matrix of the corresponding eigenvectors. As it is typically done in proximal gradient algorithms~\cite{parboy13}, starting with $X_0 \DefinedAs X^k$, we obtain $X^{k+1}$ by repeating inner iterations until the desired accuracy is reached.

The above described method involves an eigenvalue decomposition in each inner iteration of the $X$-minimization problem, which requires $O(n^3)$ operations. Therefore, if the $X$-minimization step takes $q$ inner iterations to converge, a single outer iteration of ADMM requires $O(q n^3)$ operations. This is in contrast to AMA where the explicit update of $X$ takes $O(n^3)$ operations.  {Thus, ADMM and AMA have similar computational complexity; cf.~Section~\ref{sec.complexity}. However, in Section~\ref{sec.example} we demonstrate that, relative to ADMM, customized AMA provides significant speed-up via a heuristic step-size selection (i.e., a BB step-size initialization followed by backtracking).}

We finally note that, when both parts of the objective function are strongly convex, an accelerated variant of ADMM can be employed~\cite{golodosetbar14}. However, the presence of the nuclear norm in~\eqref{eq.CCP} prevents us from using such techniques. For weakly convex objective functions, restart rules in conjunction with acceleration techniques can be used to reduce oscillations that are often encountered in first-order iterative methods~\cite{odocan13,golodosetbar14}.  {Since our computational experiments do not suggest a significant improvement using restart rules, we refrain from further discussing this variant of ADMM in Section~\ref{sec.example}.
}


\subsection{AMA as a proximal gradient on the dual}
	\label{sec.AMApg}

In the follow up section, Section~\ref{sec.convergence}, we show that the gradient of the dual objective function over a convex domain is Lipschitz continuous. In the present section, we denote a bound on the Lipschitz constant by $L$, and prove that AMA with step-size $\rho = 1/L$ works as a proximal gradient on the dual problem. 
This implies that~\eqref{eq.AMA_dualupdate_1} and~\eqref{eq.AMA_dualupdate_2} are equivalent to the updates obtained by applying the proximal gradient algorithm to~\eqref{eq.dual}.

{
The dual problem~\eqref{eq.dual} takes the following form
\begin{align}
\label{eq.prox_general}
	\ba{cl}
	\minimize\limits_{Y_1,Y_2}
	&\!\!
	f(Y_1,Y_2)
	\,+\,
	g(Y_1,Y_2)
	 \ea
\end{align}
where $f(Y_1,Y_2) = -\logdet \cA^\dagger(Y_1, Y_2) - \inner{G}{Y_2}$ and $g(Y_1, Y_2)$ denotes the indicator function
\[
	\cI(Y_1)
	\;=\;
	\left\{
	\ba{rl}
		0,~
		&
		\norm{Y_1}_2 \leq \gamma
		\\
		+\infty,~
		&
		\mathrm{otherwise}.
	\ea
	\right.
\]
Both $f$: $(\bbC^{n\times n},\bbC^{p\times p}) \rightarrow \bbR$ and $g$: $(\bbC^{n\times n},\bbC^{p\times p}) \rightarrow \bbR \cup \{+\infty\}$ are closed proper convex functions and $f$ is continuously differentiable. 
For  {$Y_1 \in \bbC^{n \times n}$ and $Y_2 \in \bbC^{p \times p}$}, the proximal operator of $g$, $\mathrm{prox}_g$: $(\bbC^{n\times n},\bbC^{p\times p}) \rightarrow (\bbC^{n\times n},\bbC^{p\times p})$ is given by
\begin{align}
	\ba{rcl}
	\mathrm{prox}_g (V_1, V_2)
	&\!\! = \!\!&
	\argmin\limits_{Y_1, Y_2}\,
	g(Y_1,Y_2)
	\;+\;
	\dfrac{1}{2}\, \ds{\sum_{i \, = \, 1}^2} \, \norm{Y_i \, - \, V_i}_F^2
	\ea
	\non
\end{align}
where $V_1$ and $V_2$ are fixed matrices. For~\eqref{eq.prox_general}, the proximal gradient method~\cite{parboy13} determines the updates as
 {
\[
	\left(Y_1^{k+1},\, Y_2^{k+1}\right)
	\;\DefinedAs\;
	\mathrm{prox}_{\rho g} 
	\!
	\left(Y_1^k \,-\, \rho \, \nabla_{Y_1} f (Y_1^k, Y_2^k),\, Y_2^k \,-\, \rho \, \nabla_{Y_2} f (Y_1^k, Y_2^k) \right)
\]
}
where $\rho>0$ is the step-size. For $\rho \in (0,1/L]$ this method converges with rate $O(1/k)$~\cite{comwaj05}.
}

{
Application of the proximal gradient method to the dual problem~\eqref{eq.prox_general} yields
\begin{subequations}
\label{eq.prox-update}
	\begin{eqnarray}
		Y_1^{k+1}
		&\!\! \DefinedAs \!\! &
		\argmin\limits_{Y_1} \,
		\left< \nabla_{Y_1} (-\logdet \cA^\dagger(Y_1^k, Y_2^k)), Y_1\right> 
		\,+\;
		\cI \left( Y_1 \right) 
		\,+\; 
		\dfrac{L}{2} \, \norm{Y_1 \,-\, Y_1^k}_F^2
		\label{eq.prox-update-Y1}
		\\
		Y_2^{k+1}
		&\!\! \DefinedAs \!\! &
		\argmin\limits_{Y_2} \,
		\left< \nabla_{Y_2} (-\logdet \cA^\dagger(Y_1^k, Y_2^k)), Y_2\right> 
		\,+\,
		\left< G, Y_2 \right> 
		\,+\;
		\dfrac{L}{2} \, \norm{Y_2 \,-\, Y_2^k}_F^2 
		\label{eq.prox-update-Y2}
	\end{eqnarray}
\end{subequations}
The gradient in~\eqref{eq.prox-update-Y1} is determined by
\begin{align}
	\ba{rcl}
		\nabla_{Y_1} (-\logdet \cA^\dagger(Y_1^k, Y_2^k))
		&\!\!=\!\!&
		-\cA_1(\cA^\dagger(Y_1^k, Y_2^k)^{-1})
	\ea
	\non
\end{align}
and we thus have
\be
	\label{eq.prox_Y1}
	\ba{rcl}
	Y_1^{k+1}
	&\!\! \DefinedAs \!\!&
	\argmin\limits_{Y_1} 
	\;
	\cI \left(Y_1\right)
	\; + \;
	\dfrac{L}{2}\, \norm{Y_1 \, - \, ( Y_1^k \, + \, \dfrac{1}{L}\, \cA_1(\cA^\dagger(Y_1^k, Y_2^k)^{-1}) )}_F^2.
	\ea
\ee
Since $X^{k+1} = \cA^\dagger(Y_1^k, Y_2^k)^{-1}$, it follows that the dual update $Y_1^{k+1}$ given by~\eqref{eq.AMA_dualupdate_1} solves~\eqref{eq.prox_Y1} with $\rho = 1/L$. This is because the saturation operator $\cT_\gamma$ represents the proximal mapping for the indicator function $\cI \left( Y_1 \right)$~\cite{parboy13}. Finally, using the first order optimality conditions for~\eqref{eq.prox-update-Y2} it follows that the dual update
\[
Y_2^{k+1} \;=\; Y_2^{k} + \frac{1}{L} ( \cA_2(\cA^\dagger(Y_1^k,Y_2^k)^{-1}) - G )
\]
is equivalent to~\eqref{eq.AMA_dualupdate_2} with $\rho = 1/L$.
}

\subsection{Convergence analysis}
\label{sec.convergence}
	
In this section we use the equivalence between AMA and the proximal gradient algorithm (on the dual problem) to prove convergence of our customized AMA. 
Before doing so, we first establish Lipschitz continuity of the gradient of the logarithmic barrier in the dual objective function over a pre-specified convex domain, and show that the dual iterates are bounded within this domain.
These two facts allow us to establish sub-linear convergence of AMA for~\eqref{eq.CCP}. Proofs of all technical statements presented here are provided in the appendix.

{
We define the ordered pair $Y = (Y_1, Y_2) \in \bbH^n \times \bbH^p$ where
\[
	\bbH^n \times \bbH^p
	\;=\;
	\big{\{} (Y_1,Y_2) ~|~ Y_1 \in \bbH^n ~~\mathrm{and}~~ Y_2 \in \bbH^p \big{\}},
\]
with $\bbH^n$ denoting the set of Hermitian matrices of dimension $n$. We also assume the existence of an optimal solution $\bar{Y} = (\bar{Y}_1, \bar{Y}_2)$ which is a fixed point of the dual updates~\eqref{eq.AMA_dualupdate_1} and~\eqref{eq.AMA_dualupdate_2}, i.e.,
\[
	\ba{rcl}
	\bar{Y}_1
	&\!\!=\!\!& 
	\cT_{\gamma} \left(\bar{Y}_1 \;+\; \rho\, \cA_1 (\bar{X})\right)
	\\[.15cm]
	\bar{Y}_2
	&\!\!=\!\!& 
	\bar{Y}_2 \;+\; \rho \left( \cA_2 (\bar{X}) \,-\, G \right)
	\ea
\]
where $\bar{X} = \cA^\dagger(\bar{Y})^{-1}$. Since the proof of optimality for $\bar{Y}$ follows a similar line of argument made in~\cite{dalraj14}, we refrain from including further details.
}


{
While the gradient of $J_d$ is not Lipschitz continuous over the entire domain of $\bbH^n \times \bbH^p$, we show its Lipschitz continuity over the convex domain
\be
	\label{eq.domainY}
	\cD_{\alpha\beta}
	\;=\;
	\{\,Y  \in \bbH^n \times \bbH^p ~|~ 0 < \alpha\,I \preceq \cA^\dagger(Y) \preceq \beta\, I < \infty \,\} 
\ee
for any $0<\alpha<\beta<\infty$. This is stated in the next lemma, and its proof given in the appendix relies on showing that the Hessian of $J_d$ is bounded from above.
}

{
\begin{lemma}
\label{lemma:lipschitz}
	For $Y \in \cD_{\alpha\beta}$, the function $\logdet \cA^\dagger(Y)$ has a Lipschitz continuous gradient with Lipschitz constant $L = \sigma_{\max}^2(\cA^\dagger)/\alpha^2$, where ${\sigma_{\max}(\cA^\dagger)}$ is the largest singular value of the operator $\cA^\dagger$.
\end{lemma}
}


We next show that the dual AMA iterations~\eqref{eq.AMA_dualupdate_1} and~\eqref{eq.AMA_dualupdate_2} are contractive, which is essential in establishing that the iterates are bounded within the domain $\cD_{\alpha\beta}$.

{
\begin{lemma}
\label{lemma:contractive}
	Consider the map $Y \mapsto Y^+$
	\begin{subequations}\label{eq:contractivemap}
	\begin{eqnarray}
	Y_1^+ &\!\!=\!\!& \cT_{\gamma} \left(Y_1 \;+\; \rho\, \cA_1 (\cA^\dagger(Y)^{-1}) \right)
	\\[.15cm]
	Y_2^+ &\!\!=\!\!& Y_2 \;+\; \rho \left( \cA_2 (\cA^\dagger(Y)^{-1}) \,-\, G \right),
	\end{eqnarray}
	\end{subequations}
where $Y = (Y_1, Y_2)$. Let $0<\alpha<\beta<\infty$ be such that
\[
	\alpha\,I \,\preceq\, \cA^\dagger( \bar{Y}) \,\preceq\, \beta\,I,
\]
where $\bar{Y} = (\bar{Y}_1, \bar{Y}_2)$ denotes a fixed point of \eqref{eq:contractivemap}.
Then, for any $0<\rho \leq \dfrac{2\, \alpha^4}{\beta^2\, \sigma_{\max}^2 (\cA)}$, the map~\eqref{eq:contractivemap} is contractive over $\cD_{\alpha\beta}$, that is, 
	\[
	\norm{Y^+ \,-\, \bar{Y}}_F 
	\; \leq \;
	\norm{Y \,-\, \bar{Y}}_F
	~~~
	\mbox{for any $Y \, \in \, \cD_{\alpha\beta}$.}
	\]
\end{lemma}
}


{
As noted above, it follows that the dual AMA iterates $\{Y^k\}$ belong to the domain $\cD_{\alpha\beta}$. This is stated explicitly next in Lemma \ref{lemma:iterates}.
In fact, the lemma establishes universal lower and upper bounds on $\cA^\dagger(Y^k)$ for all $k$. These bounds guarantees that the dual iterates $\{Y^k\}$ belong to the domain $\cD_{\alpha\beta}$ and that Lipschitz continuity of the gradient of the dual function is preserved through the iterations. 
}

{
\begin{lemma}
\label{lemma:iterates}
	Given a feasible initial condition $Y^0$, i.e., $Y^0$ satisfies $\cA^\dagger(Y^0)\succ 0$ and $\|Y_1^0\|_2\le \gamma$, let $\alpha$, $\beta > 0$
\[
	\ba{rcl}
		\beta
		&\!\!=\!\!&
		\sigma_{\max}(\cA^\dagger)\,\norm{Y^0 \,-\, \bar{Y}}_F \;+\; \norm{\cA^\dagger(\bar{Y})}_2
		\\[.15cm]
		\alpha
		&\!\!=\!\!&
		\det \cA^\dagger(Y^0)\, \beta^{1-n}\, \mre^{-\,\left<G, Y_2^0\right> \,-\; \gamma \sqrt{n}\,\sigma_{\max} (\cA_1^\dagger)\, \trace(\bar{X})}.
	\ea
\]	
Then, for any positive step-size $\rho \leq \dfrac{2\, \alpha^4}{\beta^2\, \sigma_{\max}^2 (\cA)}$, we have
\[
	\alpha \, I
	\;\preceq\;
	\cA^\dagger(Y^k)
	\;\preceq\;
	\beta\, I
	~~~
	\mbox{for all $k\ge 0$.}
\] 
\end{lemma}
}


{ 
Since AMA works as a proximal gradient on the dual problem its convergence properties follow from standard theoretical results for proximal gradient methods~\cite{comwaj05}. In particular, it can be shown that the proximal gradient algorithm with step-size $\rho = 1/L$ ($L$ being the Lipschitz constant in Lemma~\ref{lemma:lipschitz}) falls into a general family of {\em majorization-minimization} algorithms for which convergence properties are well-established~\cite{hunlan04}.}

{
The logarithmic barrier in the dual function is convex and continuously differentiable. Furthermore, its gradient is Lipschitz continuous over the domain $\cD_{\alpha\beta}$. Therefore, starting from the pair $Y^0 = (Y_1^0, Y_2^0)$ a positive step-size $\rho \leq \min \Big{\{} \dfrac{2\,\alpha^4}{\beta^2\, \sigma^2_{\max}(\cA)}, \dfrac{\alpha^2}{\sigma^2_{\max}(\cA^\dagger)} \Big{\}}$ guarantees that $\{Y^k\}$ converges to $\bar{Y}$ at a sub-linear rate that is no worse than $O(1/k)$,
\[
	J_d(Y^k) \;-\; J_d(\bar{Y})
	\; \leq \; O(1/k).
\]
Since $\cA^\dagger$ is not an invertible mapping, $-\logdet \cA^\dagger(Y)$ cannot be strongly convex over $\cD_{\alpha\beta}$. Thus, in general, AMA with a constant step-size cannot achieve a linear convergence rate~\cite{cheroc97,nes07}. In computational experiments, we observe that a heuristic step-size selection (BB step-size initialization followed by backtracking) can improve the convergence of AMA; see Section~\ref{sec.example}.
}




	\vspace*{-2ex}
\section{Computational experiments}
\label{sec.example}

We provide an example to demonstrate the utility of our modeling and optimization framework. This is based on a stochastically-forced mass-spring-damper (MSD) system.
Stochastic disturbances are generated by a low-pass filter,
\begin{subequations}\label{eq.MSD}
	\be
	\ba{rccl}
	\text{low-pass filter:}
	& \dot{\zeta} 
	& \!\! = \!\! & 
	-\zeta \; + \; d \ea
	\label{eq.MSDa}
	\ee
	where $d$ represents a zero-mean unit variance white process.
	The state space representation of the MSD system is given by
\be\ba{rccl}
	\text{MSD system:}
	& 
	\dot{x} 
	& \!\! = \!\! &  
	A \, x \; + \; B_{\zeta} \, \zeta
	\ea
	\label{eq.MSDb}
	\ee
	\end{subequations}
where the state vector $x = [\,p^*\, ~v^*\,]^*$, contains position and velocity of masses. Accordingly, the state and input matrices are
\begin{align}
	A \; = \, \tbt{O}{I}{-T}{-I},
	~~~~ 
	B_{\zeta} \; = \, \tbo{0}{I}
	\non
\end{align}
where $O$ and $I$ are zero and identity matrices of suitable sizes, and $T$ is a symmetric tridiagonal Toeplitz matrix with $2$ on the main diagonal and $-1$ on the first upper and lower sub-diagonals.


The steady-state covariance of system~\eqref{eq.MSD} can be found as the solution to the Lyapunov equation
\be
	\tilde{A}\, \Sigma \;+\; \Sigma\, \tilde{A}^* \; + \;  \tilde{B} \, \tilde{B}^* \; = \; 0
	\label{eq.MSD_lyap}
	\non
\ee
where
\be
	\tilde{A} \,=\, \tbt{A}{B_{\zeta}}{O}{-I},~~~~ \tilde{B} \,=\, \tbo{0}{I}
	\non
\ee
and
\be
	\Sigma \,=\, \tbt{\Sigma_{xx}}{\Sigma_{x \zeta}}{\Sigma_{\zeta x}}{\Sigma_{\zeta \zeta}}.
	\non
\ee
The matrix $\Sigma_{xx}$ denotes the state covariance of the MSD system, partitioned as,
\be
	\Sigma_{xx} \,=\, \tbt{\Sigma_{pp}}{\Sigma_{pv}}{\Sigma_{vp}}{\Sigma_{vv}}.
	\non
\ee
We assume knowledge of one-point correlations of the position and velocity of masses, i.e., the diagonal elements of matrices $\Sigma_{pp}$, $\Sigma_{vv}$, and $\Sigma_{pv}$. Thus, in order to account for these available statistics, we seek a state covariance $X$ of the MSD system which agrees with the available statistics.
 

 {
Additional information about our computational experiments, along with {\sc Matlab} source codes, \mbox{can be found at:}
\vspace*{-4ex}
\begin{center}
\url{http://www.ece.umn.edu/users/mihailo/software/ccama/}
\end{center}
}

 {
Recall that in~\eqref{eq.CCP}, $\gamma$ determines the importance of the nuclear norm relative to the logarithmic barrier function. While larger values of $\gamma$ yield solutions with lower rank, they may fail to provide reliable completion of the ``ideal'' state covariance $\Sigma_{xx}$. For various problem sizes, minimum error in matching $\Sigma_{xx}$ is achieved with $\gamma\approx1.2$ and for larger values of $\gamma$ the error gradually increases. For MSD system with $50$ masses, Fig.~\ref{fig.hpath_error} 
shows the relative error in matching $\Sigma_{xx}$ as a function of $\gamma$. The smallest error is obtained for $\gamma=1.2$, but this value of $\gamma$ does not yield a low-rank input correlation $Z$. For $\gamma=2.2$ reasonable matching is obtained ($82.7\%$ matching) and the resulting $Z$ displays a clear-cut in its singular values with $62$ of them being nonzero; see Fig.~\ref{fig.Zsvd}.
}

\begin{figure}
\begin{center}
\begin{tabular}{cccc}
	\hspace{-.4cm}
	\begin{tabular}{c}
		\\[-.8cm]
		{\footnotesize \rotatebox{90}{\small{$\norm{X-\Sigma_{xx}}_F/\norm{\Sigma_{xx}}_F$}}}
	\end{tabular}
	& \hspace{-1.1cm}
	\subfloat[]{
	\begin{tabular}{c}
	         \includegraphics[width=0.28\textwidth]{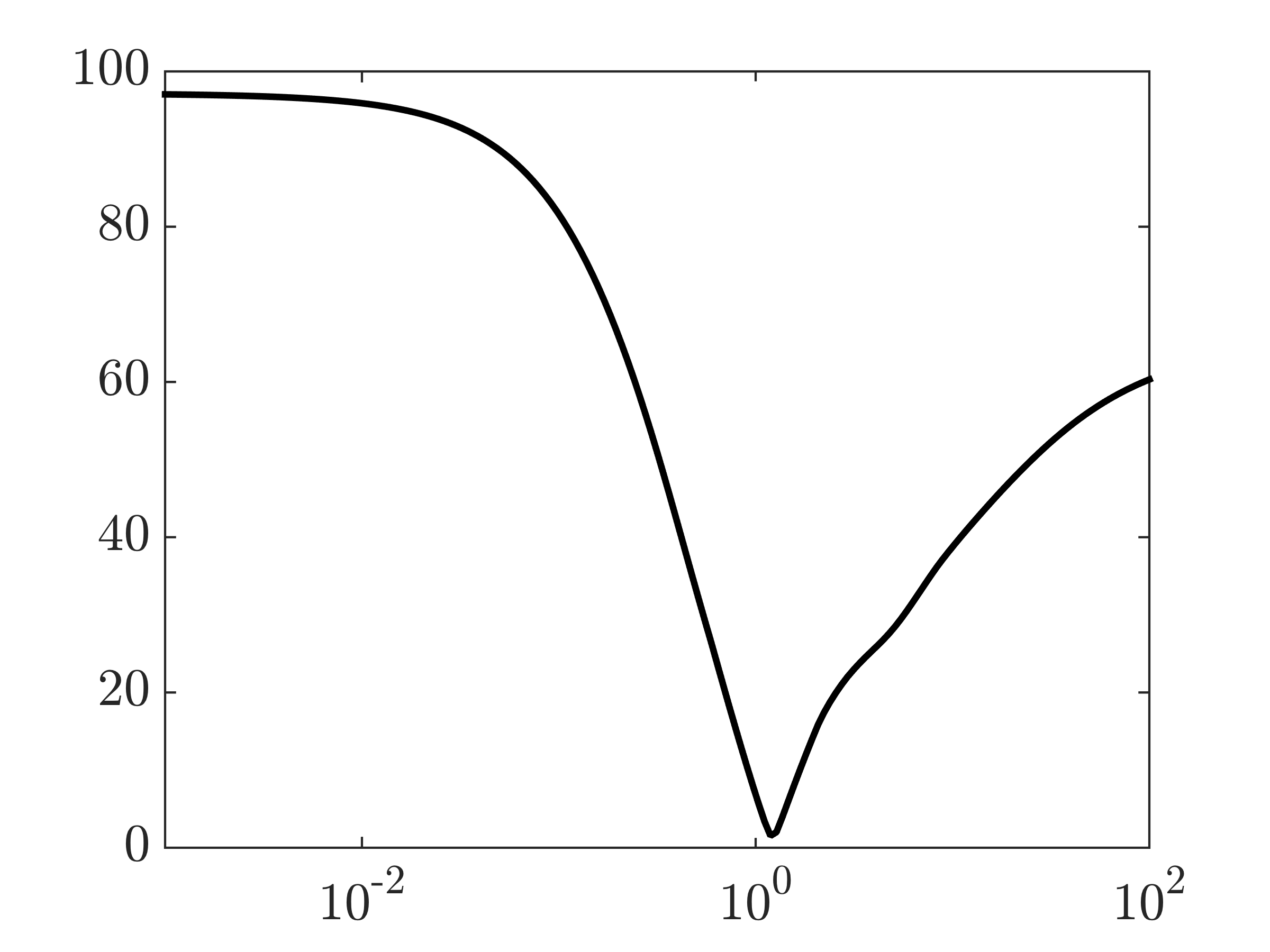}
	         \label{fig.hpath_error}
		\\[-.1cm]
		$\gamma$
	\end{tabular}
	}
	&
	\begin{tabular}{c}
		\\[-1cm]
		{\small \rotatebox{90}{$\sigma_i(Z)$}}
	\end{tabular}
	& \hspace{-1.1cm}
	\subfloat[]{
	\begin{tabular}{c}
	         \includegraphics[width=0.28\textwidth]{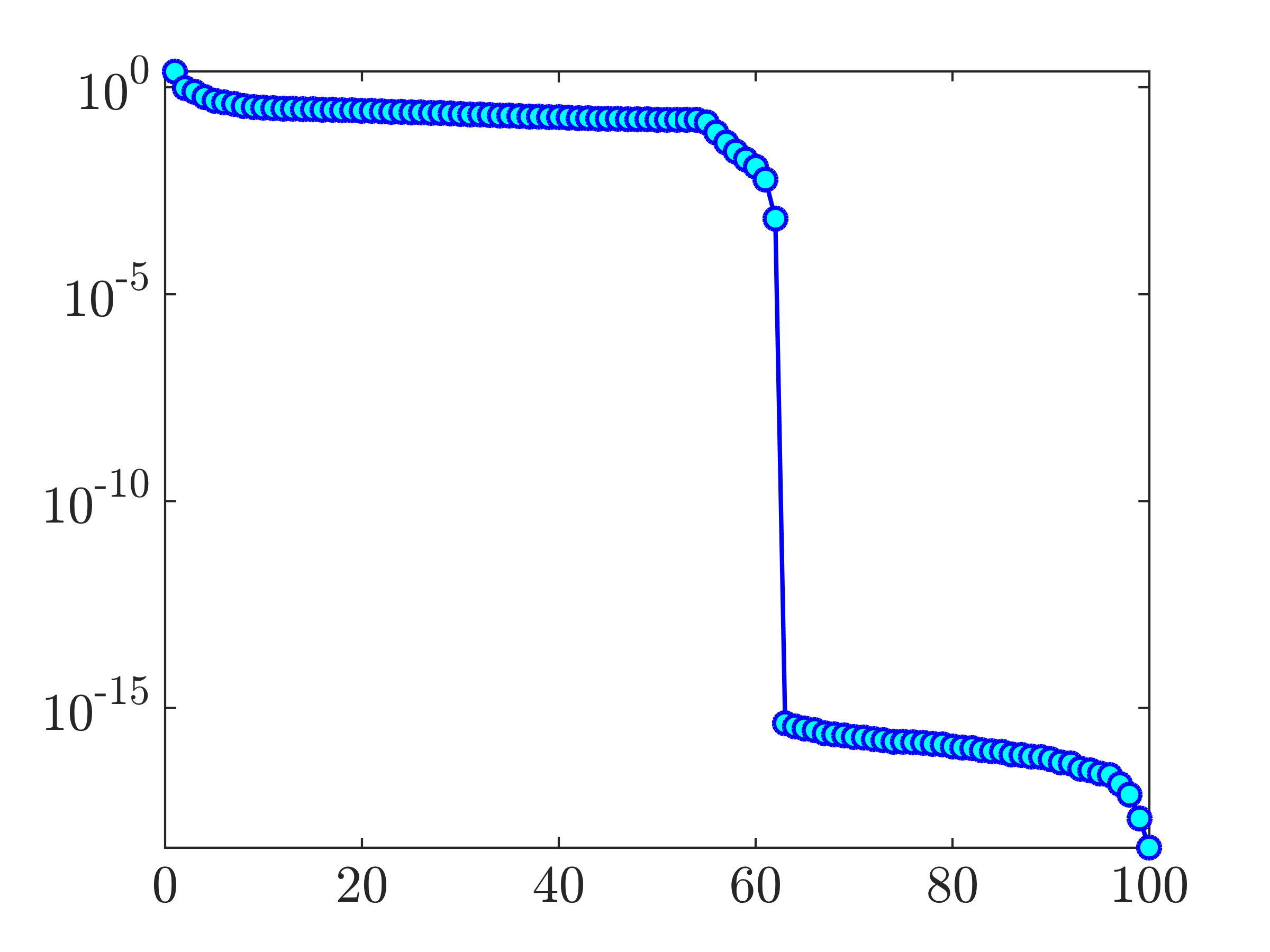}
	         \label{fig.Zsvd}
	         \\[-.1cm]
		$i$
	\end{tabular}
	}
\end{tabular}
\end{center}
\caption{(a) The $\gamma$-dependence of the relative error (percents) between the solution $X$ to~(\ref{eq.CCP})  and the true covariance $\Sigma_{xx}$ for the MSD system with $50$ masses. (b) Singular values of the solution $Z$ to~(\ref{eq.CCP}) for the MSD system with $50$ masses and $\gamma=2.2$.}
\end{figure}

%

 {
For $\gamma=2.2$, Table~\ref{table.comp1} compares solve times of CVX~\cite{cvx} and the customized algorithms of Section~\ref{sec.algorithms}.  All algorithms were implemented in {\sc Matlab} and executed on a 3.4 GHz Core(TM) i7-2600 Intel(R) machine with 12GB RAM. Each method stops when an iterate achieves a certain distance from optimality, i.e., $\norm{X^k-X^\star}_F/\norm{X^\star}_F < \eps_1$ and $\norm{Z^k-Z^\star}_F/\norm{Z^\star}_F < \eps_2$. The choice of $\eps_1, \eps_2 = 0.01$, guarantees that the primal objective is within $0.1\%$ of $J_p(X^\star,Z^\star)$. For $n=50$ and $n=100$, CVX ran out of memory. Clearly, for large problems, AMA with BB step-size initialization significantly outperforms both regular AMA and ADMM.
}



\begin{figure}
\centering
\begin{tabular}{ccc}
	\subfloat[$J_d(Y_1, Y_2)$]
	{
	\begin{tabular}{c}
	         \includegraphics[width=0.28\textwidth]{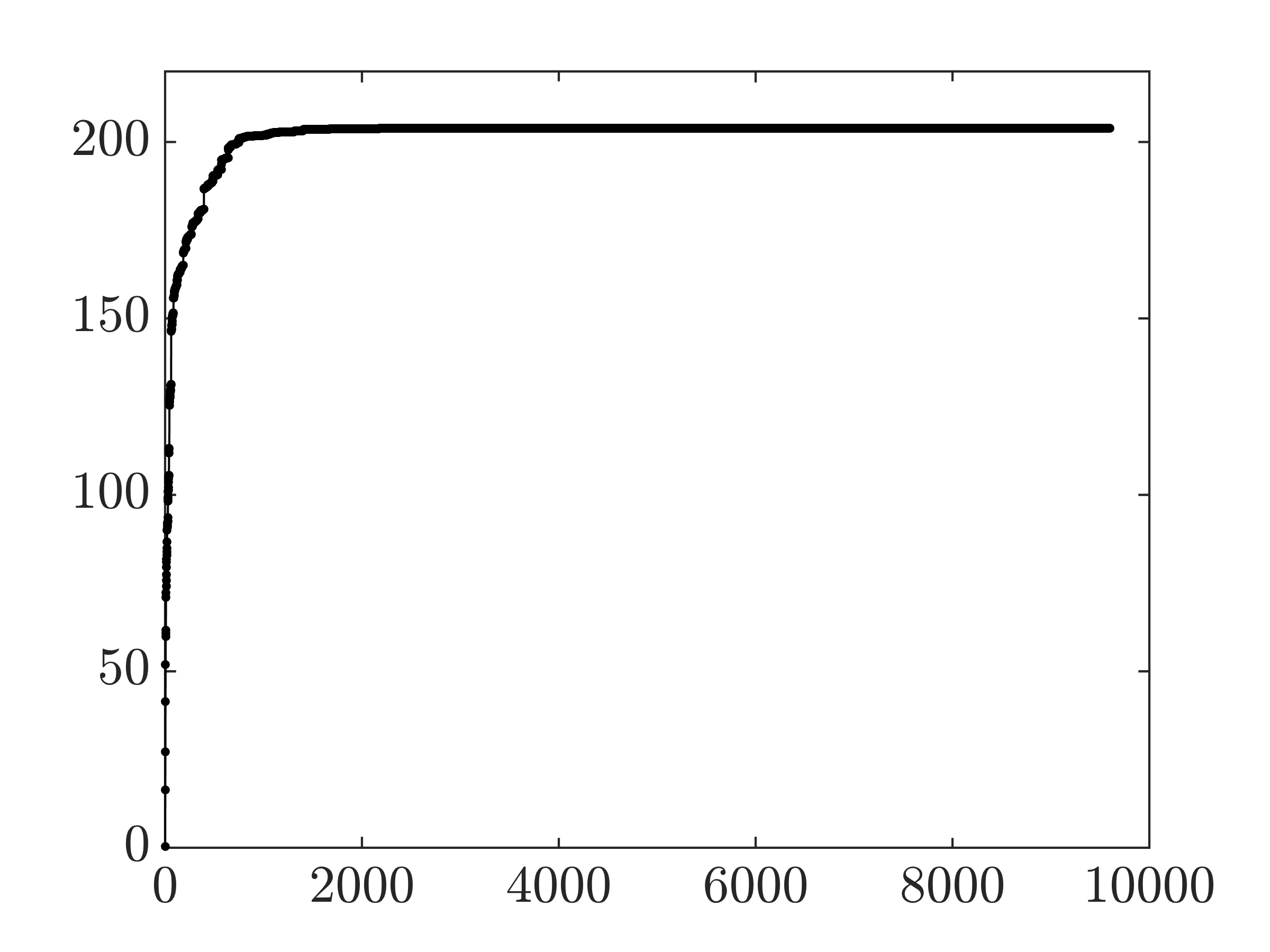}
	         \label{fig.AMAbehavior}
		\\ [-0.1cm]
		iteration
	\end{tabular}}
	& \hspace{-1cm}
	\subfloat[$\Delta_{\mathrm{gap}}$]
	{
	\begin{tabular}{c}
		\includegraphics[width=0.28\textwidth]{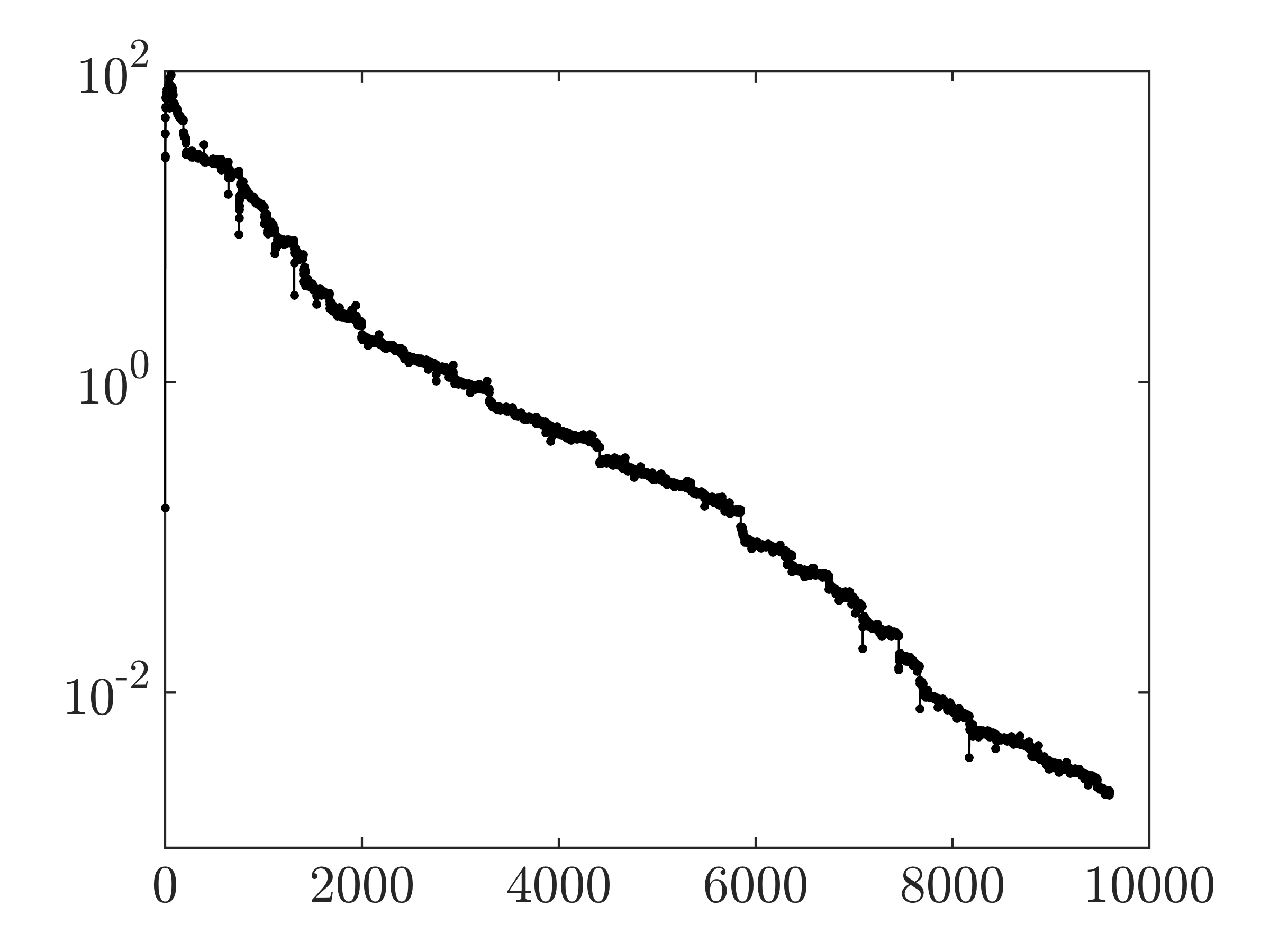}
		\label{fig.dualitygap}
		\\ [-0.1cm]
		iteration
	\end{tabular}
	}
	& \hspace{-1cm}
	\subfloat[$\Delta_{\mathrm{p}}$]
	{
	\begin{tabular}{c}
		\includegraphics[width=0.28\textwidth]{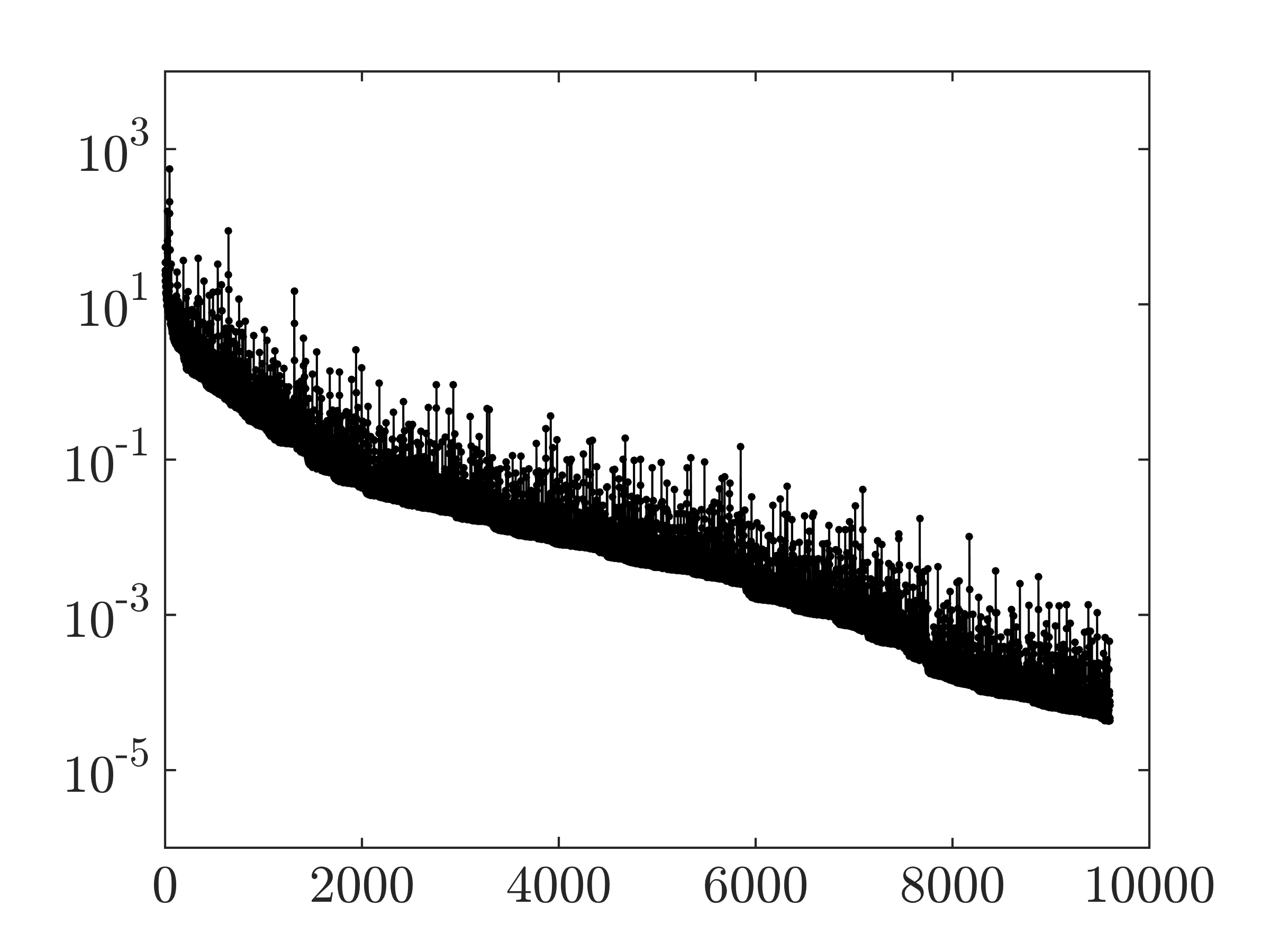}
		\label{fig.Primalresidual}
		\\ [-0.1cm]
		iteration
	\end{tabular}
	}
\end{tabular}
\caption{Performance of $\mathrm{AMA_{BB}}$ for the MSD system with $50$ masses, $\gamma=2.2$, $\eps_1 = 0.005$, and $\eps_2 = 0.05$. (a) The dual objective function $J_d(Y_1, Y_2)$ of~\eqref{eq.CCP}; (b) the duality gap, $|\Delta_{\mathrm{gap}}|$; and (c) the primal residual, $\Delta_{\mathrm{p}}$.}
\label{fig.dualitygap_primalresidual}
\end{figure}

\begin{figure}
\begin{center}
	\vspace{.2cm}
	\begin{tabular}{ccc}
	\begin{tabular}{c}
		\\[-.8cm]
		\rotatebox{90}{\scriptsize $| J_d^\star - J_d(Y_1^k, Y_2^k) |/ | J_d^\star |$}
	\end{tabular}
	& \hspace{-.8cm}
	\subfloat[]
	{
	\begin{tabular}{c}
	         \includegraphics[width=0.28\textwidth,height=0.22\textwidth]{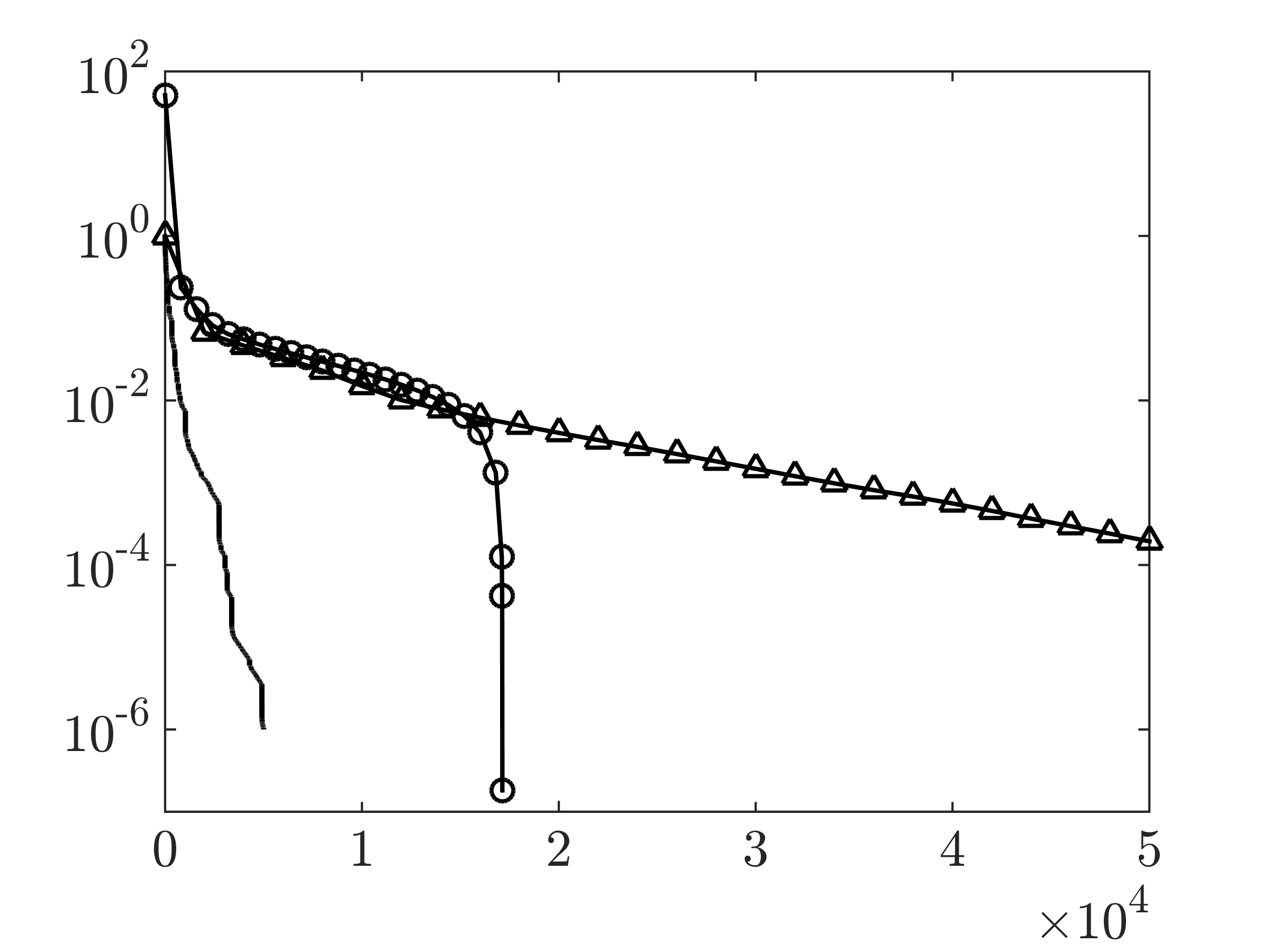}
		\\ [-0.3cm]
		iteration
	\end{tabular}
	\label{fig.error_Jd_iteration}
	}
	& \hspace{-.6cm}
	\subfloat[]
	{
	\begin{tabular}{c}
	         \includegraphics[width=0.28\textwidth]{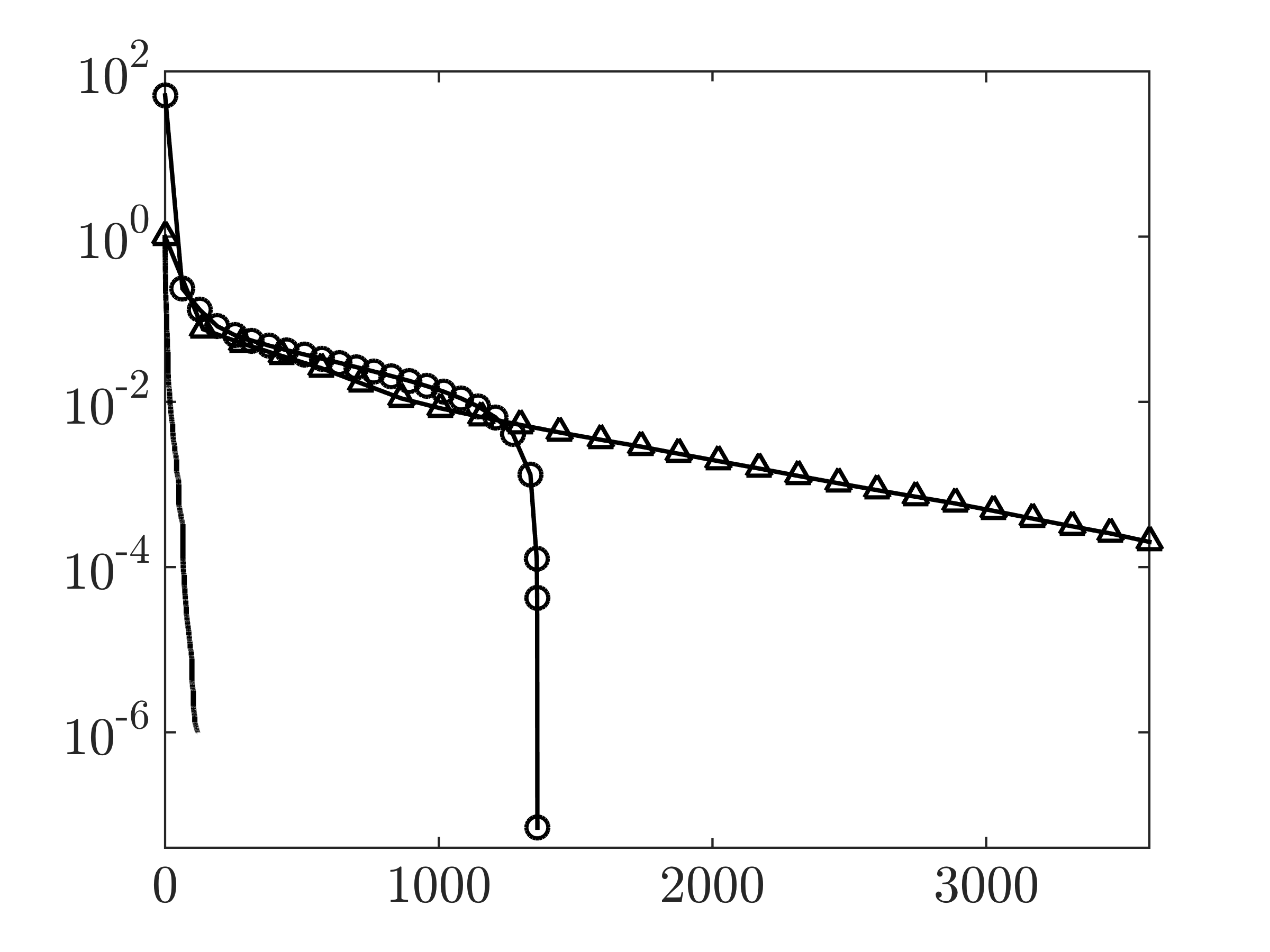}
		\\ [-0.1cm]
		solve time (seconds)
	\end{tabular}
	\label{fig.error_Jd_time}
	}
	\end{tabular}
\end{center}
\caption{ {Convergence curves showing performance of ADMM ($\circ$) and AMA with ($-$) and without ($\triangle$) BB step-size initialization vs. (a) the number of iterations; and (b) solve times for the MSD system with $50$ masses and $\gamma=2.2$. Here, $J_d^\star$ is the value of the optimal dual objective.} 
}
\label{fig.error_Jd}
\end{figure}


\begin{figure}
\centering
\begin{tabular}{cc}
\hspace{-.65cm}
\subfloat[$\diag(\Sigma_{pp})$, $\diag(X_{pp})$]
	{
\includegraphics[width=0.28\textwidth]{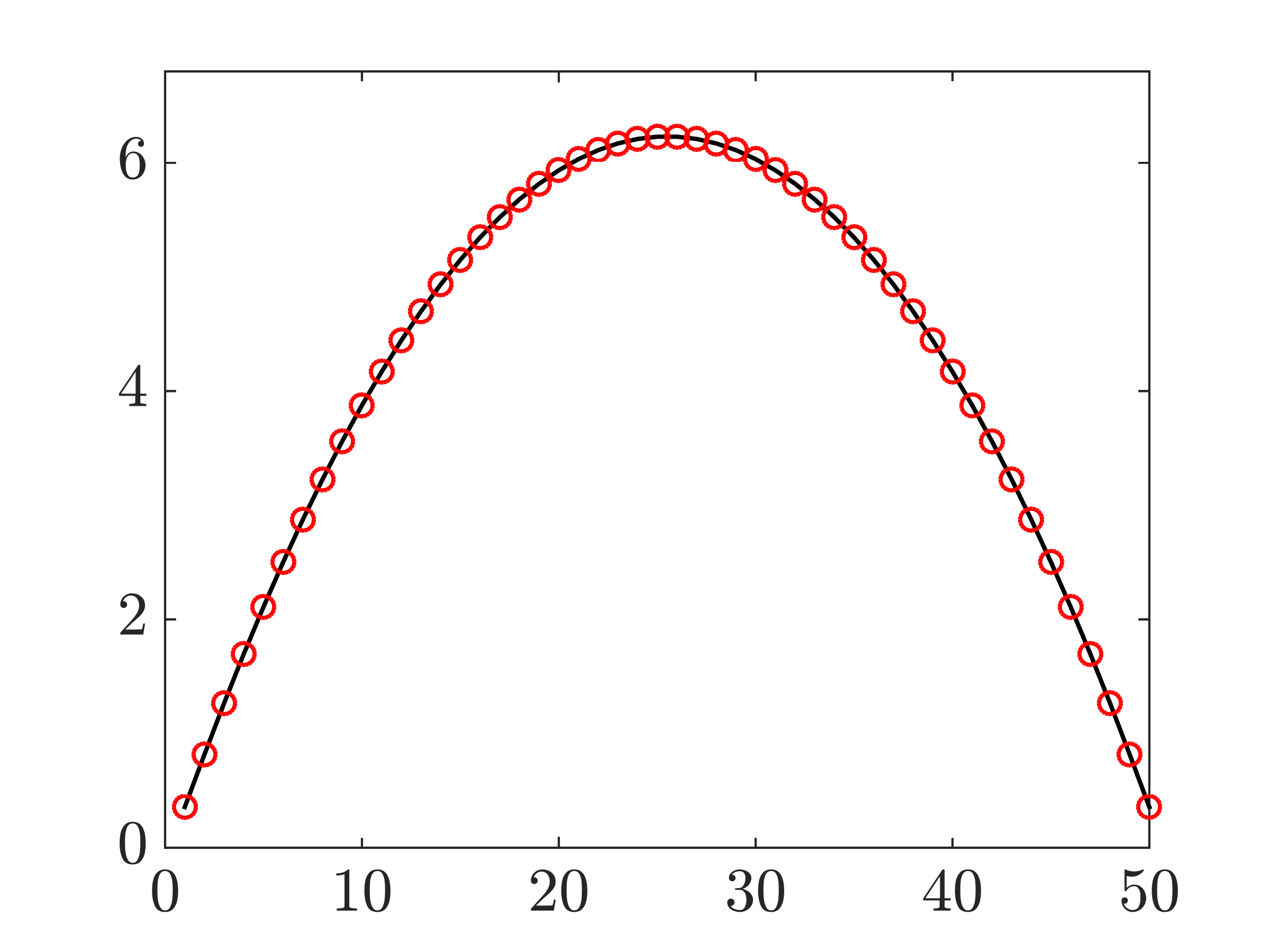}
\label{fig.diagXvv_feasibility}
	}
	& 
\subfloat[$\diag(\Sigma_{vv})$, $\diag(X_{vv})$]
	{
\includegraphics[width=0.28\textwidth]{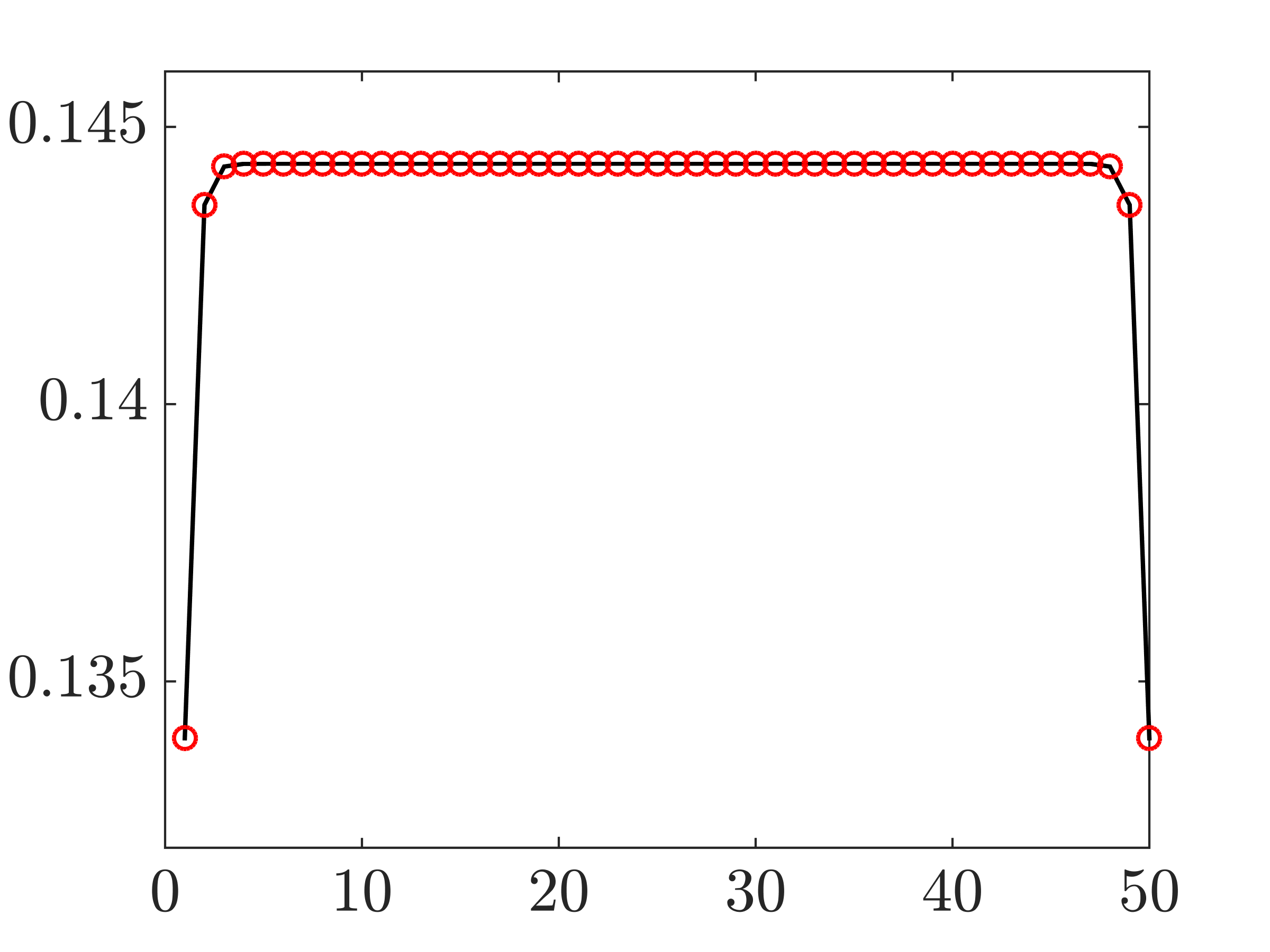}
\label{fig.diagXpp_feasibility}
	}
\end{tabular}
\caption{Diagonals of (a) position and (b) velocity covariances for the MSD system with $50$ masses; Solid black lines show diagonals of $\Sigma_{xx}$ and red circles mark solutions of optimization problem~\eqref{eq.CCP}.}
\label{fig.diag_Xpp_Xvv_feasibility}
\end{figure}


\begin{figure}
\begin{center}
	\begin{tabular}{c}
	         \includegraphics[width=0.28\textwidth]{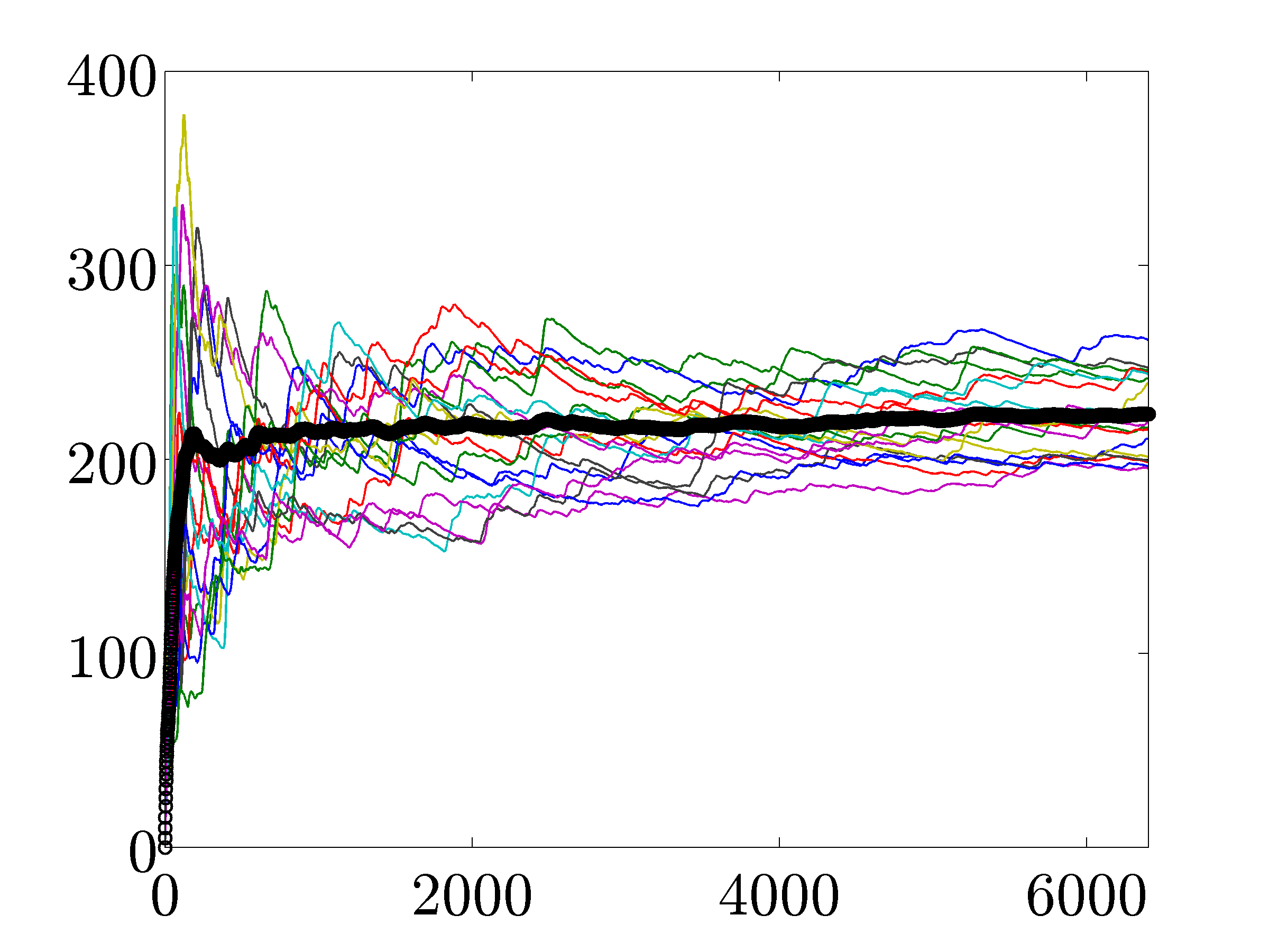}
	         \\[-.1cm]
		time
	\end{tabular}
\end{center}
\caption{Time evolution of the variance of the MSD system's state vector for twenty realizations of white-in-time forcing to~\eqref{eq.7b}. The variance averaged over all simulations is marked by the thick black line.}
\label{fig.stochastic_sim}
\end{figure}



 {
For MSD system with $50$ masses and $\gamma=2.2$, we now focus on the convergence of AMA. Figure~\ref{fig.AMAbehavior} shows monotonic increase of the dual objective function. The absolute value of the duality gap $| \Delta_{\mathrm{gap}} |$ and the primal residual $\Delta_{\mathrm{p}}$ demonstrate convergence of our customized algorithm; see Figs.~\ref{fig.dualitygap} and~\ref{fig.Primalresidual}. In addition, Fig.~\ref{fig.error_Jd_iteration} shows that regular AMA converges linearly to the optimal solution and that AMA with BB step-size initialization outperforms both regular AMA and ADMM. Thus, heuristic step-size initialization can improve the theoretically-established convergence rate. Similar trends are observed when convergence curves are plotted as a function of time; see~\ref{fig.error_Jd_time}. Finally, Fig.~\ref{fig.diag_Xpp_Xvv_feasibility} demonstrates feasibility of the optimization problem~\eqref{eq.CCP} and perfect recovery of the available diagonal elements of the covariance matrix.
}

For $\gamma=2.2$, the spectrum of $Z$ contains $50$ positive and $12$ negative eigenvalues. Based on Proposition~\ref{thm:lemma 1}, $Z$ can be decomposed into $B H^*+ H B^*$, where $B$ has $50$ independent columns. In other words, the identified $X$ can be explained by driving the state-space model with $50$ stochastic inputs $u$. The algorithm presented in Section~\ref{sec:fac} is used to decompose $Z$ into $B H^*+H B^*$. For the identified input matrix $B$, the design parameter $K$ is then chosen to satisfy the optimality criterion described in Section~\ref{sec.feedback_control}. This yields the optimal filter~\eqref{eq:linearfeedback} that generates the stochastic input $u$. We use this filter to validate our approach as explained next.

We conduct linear stochastic simulations of system~\eqref{eq.7b} with zero-mean unit variance input $w$. Figure~\ref{fig.stochastic_sim} shows the time evolution of the state variance of the MSD system. Since proper comparison requires ensemble-averaging, we have conducted twenty stochastic simulations with different realizations of the stochastic input $w$ to~\eqref{eq.7b}. The variance, averaged over all simulations, is given by the thick black line. Even though the responses of individual simulations differ from each other, the average of twenty sample sets asymptotically approaches the correct steady-state variance.


The recovered covariance matrix of mass positions $X_{pp}$ resulting from the ensemble-averaged simulations {of~\eqref{eq.7b}} is shown in Fig.~\ref{fig.Xpp}. We note that
(i) only diagonal elements of this matrix (marked by the black line) are used as data in the optimization problem~\eqref{eq.CCP}, and that
(ii) the recovery of the off-diagonal elements is remarkably consistent.
This is to be contrasted with typical matrix completion techniques that require incoherence in sampling entries of the covariance matrix.
The key in our formulation of structured covariance completion is
 the Lyapunov-like structural constraint~\eqref{eqn:Constraint on Sigma} in~\eqref{eq.CCP}.
 Indeed, it is precisely this constraint that retains the relevance of the system dynamics and, thereby, the physics of the problem.

\begin{figure}[htb!]
\centering
\begin{tabular}{cc}
\hspace{-.65cm}
\subfloat[$\Sigma_{pp}$]
	{
\includegraphics[width=0.28\textwidth]{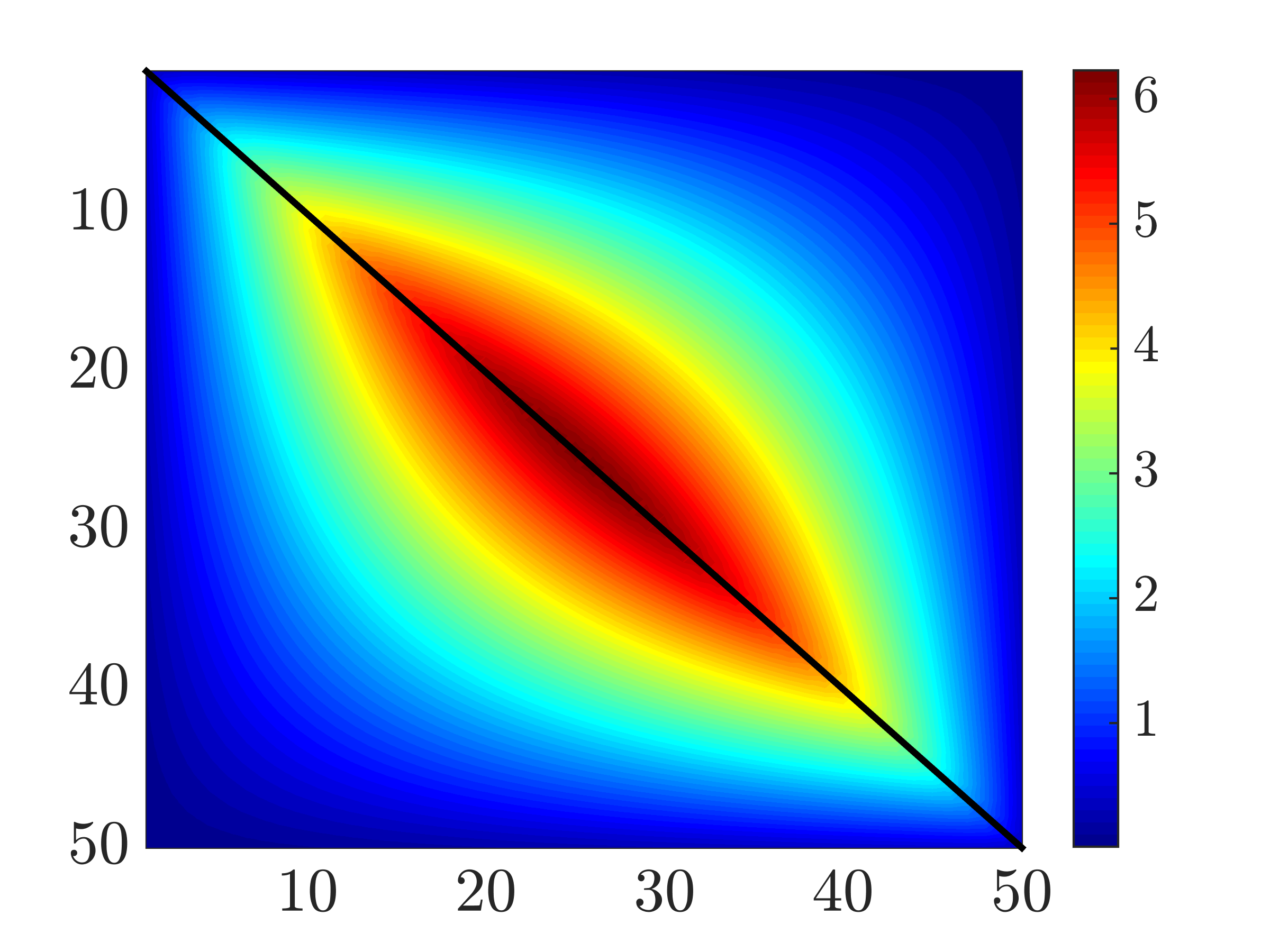}
\label{fig.Sigmapp}
	}
	& 
\subfloat[$X_{pp}$]
	{
\includegraphics[width=0.28\textwidth]{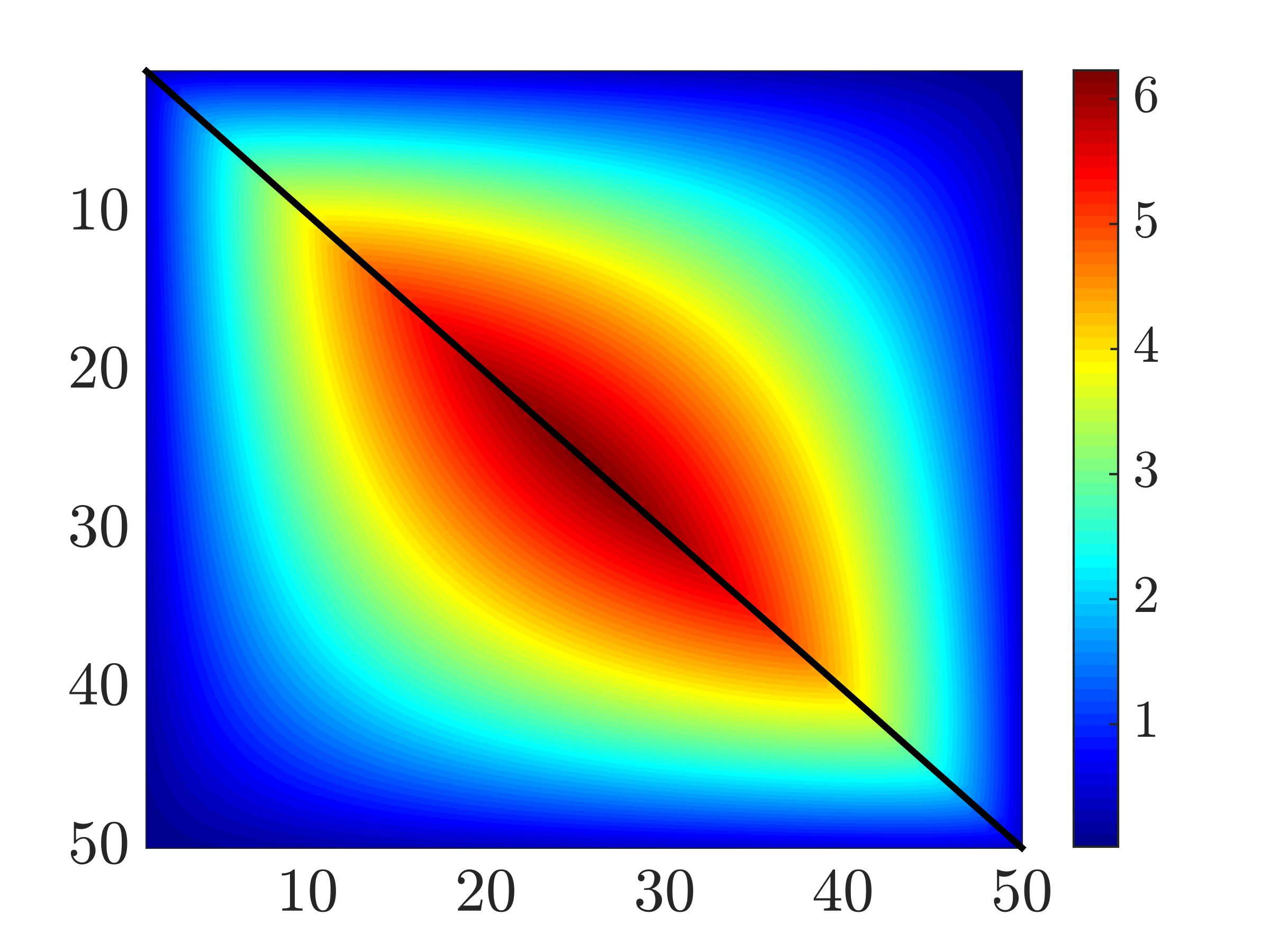}
\label{fig.Xpp}
	}
\end{tabular}
\caption{The true covariance $\Sigma_{pp}$ of the MSD system and the covariance $X_{pp}$ resulting from linear stochastic simulations {of~\eqref{eq.7b}. Available one-point correlations of the position of masses used in~\eqref{eq.CCP} are marked by black lines along the main diagonals.}}
\label{fig.Sigmapp_Xpp}
\end{figure}

\begin{table}
\vspace{.2cm}
\centering
\caption{ {Solve times (in seconds) for different number of masses and $\gamma=2.2$.}}
\label{table.comp1}
 {
\begin{tabular}{ | c | c | c | c | c |}
\hline
\!\!\! $\mathrm{n}$ \!\!\! & \!\!\! $\mathrm{CVX}$ \!\!\! & \!\!\! $\mathrm{ADMM}$ \!\!\! & \!\!\! $\mathrm{AMA}$ \!\!\! & \!\!\! $\mathrm{AMA_{BB}}$ \!\!\!
\\
\hline
\hline
\!\!\! $10$ \!\!\! & \!\!\! $28.4$ \!\!\! & \!\!\! $2$ \!\!\! & \!\!\! $1.3$ \!\!\! & \!\!\! $0.5$ \!\!\!
\\[.1cm]
\!\!\! $20$ \!\!\!\! & \!\!\!\! $419.7$ \!\!\!\! & \!\!\!\! $54.7$ \!\!\!\! & \!\!\!\! $30.7$ \!\!\!\!\! & \!\!\! $2.2$ \!\!\!
\\[.1cm]
\!\!\! $50$ \!\!\!\! & --
& \!\!\!\! $3442.9$ \!\!\!\! & \!\!\!\! $3796.7$ \!\!\!\!\! & \!\!\! $52.7$ \!\!\!
\\[.1cm]
\!\!\! $100$ \!\!\!\! & -- & \!\!\!\! $40754$ \!\!\!\! & \!\!\!\! $34420$ \!\!\!\!\! & \!\!\! $5429.8$ \!\!\!
\\[.1cm]\hline
\end{tabular}
}
\end{table}

	 \vspace*{-4ex}
\section{Concluding remarks}
\label{sec.conclusion}

We are interested in explaining partially known second-order statistics that originate from experimental measurements or simulations using stochastic linear models. This is motivated by the need for control-oriented models of systems with large number of degrees of freedom, e.g., turbulent fluid flows. In our setup, the linearized approximation of the dynamical generator is known whereas the nature and directionality of disturbances that can explain partially observed statistics are unknown. We thus formulate the problem of identifying appropriate stochastic input that can account for the observed statistics {while being} consistent with the linear dynamics.

This inverse problem is {framed} as convex optimization. To this end, nuclear norm minimization is utilized to identify noise parameters of low rank and to complete unavailable covariance data. Our formulation relies on drawing a connection between the rank of a certain matrix and the number of disturbance channels into the linear dynamics. An important contribution is the development of a customized alternating minimization algorithm (AMA) that efficiently solves covariance completion problems of large size. In fact,
we show that our algorithm works as a proximal gradient on the dual problem and establish a sub-linear convergence rate for the fixed step-size. We also provide comparison with ADMM and demonstrate that AMA yields explicit updates of all optimization variables and a principled procedure for step-size selection. An additional contribution is the design of a class of linear filters that realize suitable colored-in-time excitation to account for the observed state statistics. These filters solve a non-standard stochastic realization problem with partial covariance information.

Broadly, our research program aims at developing a framework for control-oriented modeling of turbulent flows~\cite{zarjovgeoACC14, zarjovgeoCTR14, zarjovgeoJFM16}. The present work represents a step in this direction in that it provides a theoretical and algorithmic approach 
 {for dealing with structured covariance completion problems of sizes that arise in fluids applications. In fact, we have recently employed our framework to model second-order statistics of turbulent flows via stochastically-forced linearized Navier-Stokes equations~\cite{zarjovgeoJFM16}.}

	\vspace*{-3ex}
\section*{Appendix}
\label{sec.appendix}

	\vspace*{-3ex}
\subsection*{Proof of Lemma~\ref{thm:lemma 0}}
Without loss of generality, let us consider $Z$ of the following form
{
(see Section \ref{sec:fac} for further justification)
}
    \begin{equation}\label{eq:Zcanonical}
    Z
    \;=\;
          2 \left[
          \begin{array}{ccc}
          I_{\pi} & 0 & 0\\
          0 & -I_{\nu} & 0\\
          0 & 0 & 0
          \end{array}
          \right].
    \end{equation}
Given any $S$ that satisfies $Z=S+S^*$ we can decompose it into
    \begin{equation}
    \non
    S
    \;=\;
    M\,+\,N
    \end{equation}
with $M$ Hermitian and $N$ skew-Hermitian. It is easy to see that
    \begin{equation}\nonumber
    M
    \;=\;
    \dfrac{1}{2}\, Z
    \;=\,
    \left[
    \begin{array}{ccc}
          I_{\pi} & 0 & 0\\
          0 & -I_{\nu} & 0\\
          0 & 0 & 0
    \end{array}
    \right].
    \end{equation}
By partitioning $N$ as
    \begin{equation}\nonumber
    N
    \;=\,
    \left[
    \begin{array}{ccc}
      N_{11} & N_{12} & N_{13}\\
      N_{21} & N_{22} & N_{23}\\
      N_{31} & N_{32} & N_{33}
    \end{array}
    \right],
    \end{equation}
we have
    \begin{equation}\nonumber
    S
    \;=\,
    \left[
    \begin{array}{ccc}
     I_{\pi}+N_{11} & N_{12} & N_{13}\\
     N_{21} & -I_{\nu}+N_{22} & N_{23}\\
     N_{31} & N_{32} & N_{33}
    \end{array}
    \right].
    \end{equation}
Clearly,
    \begin{equation}\nonumber
    \rank(S)
    \;\geq\;
     \rank(I_{\pi}+N_{11}).
    \end{equation}
Since $N_{11}$ is skew-Hermitian, all its eigenvalues are on the imaginary axis. This implies that all the eigenvalues of $I_{\pi}+N_{11}$ have real part 1 and therefore $I_{\pi}+N_{11}$ is a full rank matrix. Hence, we have
\[
    \rank(S)
    \;\geq\;
     \rank(I_{\pi}+N_{11})
     \;=\;
     \pi(Z)
\]
which completes the proof.

	\vspace*{-3ex}
\subsection*{Proof of Proposition~\ref{thm:lemma 1}}

The inequality
\[
    \min\{\rank(S)~|~Z\,=\,S+S^*\}
    \;\geq\;
    \max\{\pi(Z),\nu(Z)\}
\]
follows from Lemma~\ref{thm:lemma 0}.
To establish the proposition we need to show that the bounds are tight, i.e.,
\[
    \min\{\rank(S)~|~Z\,=\,S+S^*\}
    \;\leq\;
    \max\{\pi(Z),\nu(Z)\}.
\]
Given $Z$ in~\eqref{eq:Zcanonical}, for $\pi(Z)\leq \nu(Z)$, $Z$ can be written as
    \begin{align}
    \nonumber
    Z
    \;=\;
    2
    \left[
    \begin{array}{cccc}
          I_{\pi} & 0 & 0 & 0\\
          0 & -I_{\pi} & 0 & 0\\
          0 & 0 & -I_{\nu-\pi} & 0\\
          0 & 0 & 0 & 0
    \end{array}
    \right].
    \end{align}
By selecting $S$ in the form~\eqref{eq:FormS} we conclude that
    \begin{eqnarray}
    \nonumber
    \rank(S)
    &\!\!=\!\!&
    \rank
    (
    \left[
    \begin{array}{cc}
    I_{\pi} & -I_{\pi}\\
    I_{\pi} & -I_{\pi}
    \end{array}
    \right]
    )
    \;+\;
    \rank(-I_{\nu-\pi})\nonumber\\
    &\!\!=\!\!&
    \pi(Z) \,+\, \nu(Z) \,-\, \pi(Z)
    \;=\;
    \nu(Z).
    \non
    \end{eqnarray}
Therefore
    \begin{equation}\nonumber
    \min\{\rank(S)~|~Z\,=\,S+S^*\}
    \;\leq\;
     \nu(Z).
    \end{equation}
Similarly, for the case $\pi(Z)> \nu(Z)$,
    \begin{equation}
    \nonumber
    \min\{\rank(S)~|~Z\,=\,S+S^*\}
    \;\leq\;
     \pi(Z).
    \end{equation}
Hence,
    \begin{equation}
    \nonumber
    \min\{\rank(S)~|~Z\,=\,S+S^*\}
    \;\leq\;
    \max\{\pi(Z),\nu(Z)\}
    \end{equation}
which completes the proof.

\subsection*{Proof of Lemma~\ref{lemma:lipschitz}}
{
The second-order approximation of $\logdet \cA^\dagger(Y)$ yields
\[
	\ba{rcl}
	\logdet \cA^\dagger(Y+\Delta Y)
	&\!\! = \!\!&
	\logdet \cA^\dagger (Y) 
	\, + \,
	\trace \left(\cA^\dagger(Y)^{-1} \cA^\dagger(\Delta Y) \right)
	\;-
	\\[.15cm]
	&\!\! \!\!&
	\dfrac{1}{2} \,\trace \left( \cA^\dagger(Y)^{-1} \cA^\dagger(\Delta Y)\, \cA^\dagger(Y)^{-1} \cA^\dagger(\Delta Y) \right)
	\,+\,
	O(\|\Delta Y\|_F^2).
	\ea
\]
To show Lipschitz continuity of the gradient it is sufficient to show that the approximation to the Hessian is bounded by the Lipschitz constant $L$, i.e.,
\[
	\trace \left( \cA^\dagger(Y)^{-1} \cA^\dagger(\Delta Y)\, \cA^\dagger(Y)^{-1} \cA^\dagger(\Delta Y) \right)
	\;\leq\;
	L\, \norm{\Delta Y}_F^2.
\]
From the left-hand-side we have
\[
	\ba{rcl}
	\trace \left( \cA^\dagger(Y)^{-1} \cA^\dagger(\Delta Y)\, \cA^\dagger(Y)^{-1} \cA^\dagger(\Delta Y) \right)
	&\!\! = \!\!&
	\trace \left( \cA^\dagger(Y)^{-\frac{1}{2}}\, \cA^\dagger(\Delta Y)\, \cA^\dagger(Y)^{-1} \cA^\dagger(\Delta Y)\, \cA^\dagger(Y)^{-\frac{1}{2}} \right)
	\\[.2cm]
	&\!\! \leq \!\!&
	\dfrac{1}{\alpha}\, \trace \left( \cA^\dagger(Y)^{-\frac{1}{2}}\, \cA^\dagger(\Delta Y)\, \cA^\dagger(\Delta Y)\, \cA^\dagger(Y)^{-\frac{1}{2}} \right)
	\\[.2cm]
	&\!\! = \!\!&
	\dfrac{1}{\alpha}\, \trace \left( \cA^\dagger(\Delta Y)\, \cA^\dagger(Y)^{-1} \cA^\dagger(\Delta Y) \right)
	\\[.2cm]
	&\!\! \leq \!\!&
	\dfrac{1}{\alpha^2}\, \trace \left( \cA^\dagger(\Delta Y) \, \cA^\dagger(\Delta Y) \right)
	\\[.2cm]
	&\!\! \leq \!\!&
	\dfrac{\sigma^2_{\max}(\cA^\dagger)}{\alpha^2}\, \norm{\Delta Y}_F^2.
	\ea
\]
Here, we have repeatedly utilized the fact that $\cA^\dagger(Y)^{-1}$ is a positive-definite matrix and that $Y\in \cD_{\alpha\beta}$. This completes the proof.
}

	\vspace*{-4ex}
\subsection*{Proof of Lemma~\ref{lemma:contractive}}
	{
We begin by substituting the expressions for $Y_1^+$, $Y_2^+$, $\bar{Y}_1$, and $\bar{Y}_2$. Utilizing the non-expansive property of the proximal operator $\cT_\gamma$~\cite{roc76} we have
\[
	\ba{rcl}
	\norm{Y_1^+ -\, \bar{Y}_1}_F
	&\!\!= \!\!&
	||\; \cT_{\gamma} \left(Y_1 \,+\, \rho\, \cA_1(\cA^\dagger(Y_1, Y_2)^{-1})\right) 
	\,-\,
	\cT_{\gamma} \left(\bar{Y}_1 \,+\, \rho\, \cA_1(\cA^\dagger(\bar{Y}_1, \bar{Y}_2)^{-1})\right) ||_F
	\\[.15cm]
	&\!\! \leq \!\!&
	||\; Y_1 \,+\, \rho\, \cA_1(\cA^\dagger(Y_1, Y_2)^{-1}) 
	\,-\,
	\bar{Y}_1 \,-\, \rho\, \cA_1(\cA^\dagger(\bar{Y}_1, \bar{Y}_2)^{-1}) \,||_F
	\\[.25cm]
	\norm{Y_2^+ -\, \bar{Y}_2}_F
	&\!\! = \!\!& 
	||\; Y_2 \,+\, \rho \left( \cA_2 (\cA^\dagger(Y_1, Y_2)^{-1}) \,-\, G \right)
	\,-\,
	\bar{Y}_2 \,-\, \rho \left( \cA_2 (\cA^\dagger(\bar{Y}_1, \bar{Y}_2)^{-1}) \,-\, G \right) ||_F,
	\ea
\]
from which we obtain
\[
	\norm{Y^+ -\, \bar{Y}}_F
	\;\leq\;
	\norm{h_\rho(Y) \,-\, h_\rho(\bar{Y})}_F,
\]
where
\[
	h_\rho(Y)
	\;=\;
	Y
	\;+\;
	\rho\, \cA\left(\cA^\dagger(Y)^{-1} \right).
\]
The first order approximation of the linear map $h_\rho$ gives
\[
	\ba{l}
	h_\rho(Y+\Delta Y)
	\,-\,
	h_\rho(Y)
	\;\approx\;
	\Delta Y \,-\,  \rho\, \cA \left(\cA^\dagger(Y)^{-1} \cA^\dagger(\Delta Y) \cA^\dagger(Y)^{-1} \right).
	\ea
\]
From this we conclude that its Jacobian at $Y$ is $\le 1$ if
\[
	\left< \cA \left(\cA^\dagger(Y)^{-1} \cA^\dagger(\Delta Y) \cA^\dagger(Y)^{-1} \right), \Delta Y\right>
	\;\geq\;
	0,
\]
and the step-size $0 < \rho \leq \dfrac{2\, \alpha^4}{\beta^2\, \sigma^2_{\max}(\cA)}$ satisfies
\[
	\dfrac{\rho}{2}\, \norm{\cA \left(\cA^\dagger(Y)^{-1} \cA^\dagger(\Delta Y)\, \cA^\dagger(Y)^{-1} \right)}_F^2
	\;\leq\;
	\left< \cA \left(\cA^\dagger(Y)^{-1} \cA^\dagger(\Delta Y)\, \cA^\dagger(Y)^{-1} \right), \Delta Y\right>
\]
for any perturbation $\Delta Y$. The former follows from
\[
	\ba{rcl}
	\left< \cA^\dagger(Y)^{-1} \cA^\dagger(\Delta Y) \cA^\dagger(Y)^{-1}, \cA^\dagger(\Delta Y) \right>
	&\!\! = \!\!&
	\trace\, ( \cA^\dagger(Y)^{-1} \cA^\dagger(\Delta Y)\, \cA^\dagger(Y)^{-1} \cA^\dagger(\Delta Y))
	\\[.15cm]
	&\!\! = \!\!&
	\trace\, ( \cA^\dagger(Y)^{-1/2} \cA^\dagger(\Delta Y) \cA^\dagger(Y)^{-1} \cA^\dagger(\Delta Y) \cA^\dagger(Y)^{-1/2} ),
	\ea
\]
and the latter follows from
\[
	\ba{rcl}
	\dfrac{\rho}{2}\, \norm{\cA \left(\cA^\dagger(Y)^{-1} \cA^\dagger(\Delta Y)\, \cA^\dagger(Y)^{-1} \right)}_F^2
	 &\!\! \leq \!\!&
	\dfrac{\rho\, \sigma^2_{\max}(\cA)}{2}\, \norm{\cA^\dagger(Y)^{-1} \cA^\dagger(\Delta Y)\, \cA^\dagger(Y)^{-1}}_F^2
	\\[.25cm]
	&\!\! \leq \!\!&
	\dfrac{\rho\, \sigma^2_{\max}(\cA)}{2\,\alpha^4}\, \norm{\cA^\dagger(\Delta Y)}_F^2
	\\[.25cm]
	&\!\! \leq \!\!&
	\dfrac{\rho\, \beta^2\, \sigma^2_{\max}(\cA)}{2\, \alpha^4}\,
	\left< \cA^\dagger(Y)^{-1} \cA^\dagger(\Delta Y) \cA^\dagger(Y)^{-1}, \cA^\dagger(\Delta Y) \right>
	\\[.25cm]
	&\!\! = \!\!&
	\dfrac{\rho\, \beta^2\, \sigma^2_{\max}(\cA)}{2\,\alpha^4}\,
	\left< \cA \left(\cA^\dagger(Y)^{-1} \cA^\dagger(\Delta Y) \cA^\dagger(Y)^{-1} \right), \Delta Y\right>
	\\[.25cm]
	&\!\! \leq \!\!&
	\left< \cA \left(\cA^\dagger(Y)^{-1} \cA^\dagger(\Delta Y) \cA^\dagger(Y)^{-1} \right), \Delta Y\right>.
	\ea
\]
Thus, we conclude $J_{h_\rho}(Y)\le 1$ for all $Y\in \cD_{\alpha\beta}$.  Finally, from the mean value theorem, we have
\[
	\ba{rcl}
	\norm{Y^+ -\, \bar{Y}}_F
	&\!\! \leq \!\!&
	\norm{h_\rho(Y) \,-\, h_\rho(\bar{Y})}_F
	\\[.15cm]
	&\!\! \leq \!\!&
	\sup\limits_{\delta \,\in\, [0, 1]}\, J_{h_\rho} (Y_\delta)\, \norm{Y \,-\, \bar{Y}}_F,
	\\[.15cm]
	&\!\! \leq \!\!& \|Y \,-\, \bar{Y}\|_F
	\ea
\]
where $Y_\delta = \delta\, Y + (1-\delta)\, \bar{Y}\in \cD_{\alpha\beta}$. This completes the proof.
}

	\vspace*{-4ex}
\subsection*{Proof of Lemma~\ref{lemma:iterates}}
{
We first show that
	\[
		\alpha I 
		\,\preceq\,
		\cA^\dagger (Y^0) 
		\,\preceq\,
		 \beta I.
	\]
The upper bound follows from 
	\begin{eqnarray*}
		\|\cA^\dagger(Y^0)\|_2&=&\|\cA^\dagger(Y^0)\|_2-\|\cA^\dagger(\bar{Y})\|_2+
		\|\cA^\dagger(\bar{Y})\|_2
		\\&\le&
		\|\cA^\dagger(Y^0)-\cA^\dagger(\bar{Y})\|_2+\|\cA^\dagger(\bar{Y})\|_2
		\\&\le&
		\|\cA^\dagger(Y^0)-\cA^\dagger(\bar{Y})\|_F+\|\cA^\dagger(\bar{Y})\|_2
		\\&\le&
		\sigma_{\max}(\cA^\dagger)\|Y^0-\bar{Y}\|_F+\|\cA^\dagger(\bar{Y})\|_2
		\;=\;
		\beta.
	\end{eqnarray*}
To see the lower bound, note that for any $X \in \bbH^n$,
\be
	\label{eq.property_2norm_fro}
	\dfrac{1}{\sqrt{n}}\, \norm{X}_F 
	\;\leq\;
	\norm{X}_2 
	\; \leq \;
	\norm{X}_F.
\ee
Using this property and the dual constraint $\norm{Y_1^0}_2 \leq \gamma$ we have
\[
	\norm{\cA_1^\dagger(Y_1^0)}_2
	\; \leq \;
	\norm{\cA_1^\dagger(Y_1^0)}_F
	\;\leq\;
	\sigma_{\max} (\cA_1^\dagger)\, \norm{Y_1^0}_F
	\; \leq \;
	\sqrt{n}\, \sigma_{\max} (\cA_1^\dagger)\, \norm{Y_1^0}_2
	\;\leq\;
	\gamma\, \sqrt{n}\, \sigma_{\max} (\cA_1^\dagger).
\]
Since $\cA_1^\dagger(Y_1^0) + \cA_2^\dagger(Y_2^0) \succ 0$, we also have
\[
	\cA_2^\dagger(Y_2^0)
	\; \succeq \;
	-\, \gamma\, \sqrt{n}\, \sigma_{\max} (\cA_1^\dagger)\, I.
\]
Noting $G=\cA_2(\bar{X})$ for the optimal solution $\bar{X} \succ 0$, we obtain
\be
	\label{eq.innerproduct_bound}
	\ba{rcl}
	\left< G, Y_2^0 \right>
	\; = \;
	\left< \cA_2(\bar{X}), Y_2^0 \right>
	&\!\! = \!\!&
	\left< \bar{X}, \cA_2^\dagger(Y_2^0) \right>
	\\[.15cm]
	&\!\! \geq \!\!&
	-\, \gamma\, \sqrt{n}\, \sigma_{\max} (\cA_1^\dagger)\, \trace (\bar{X}).
	\ea
\ee
Let $a=\lambda_{\min}(\cA^\dagger(Y^0))$. Since $\beta$ gives a bound on the largest eigenvalue of $\cA^\dagger(Y^0)$, we have
\be\nonumber
	\logdet \cA^\dagger (Y^0)
	\;\leq\;
	\log(a)
	\,+\,
	(n - 1)\log(\beta)
\ee
Combining this and
\be\nonumber
	\ba{rcl}
	\logdet \cA^\dagger(Y^0)
	&\!\! \geq \!\!&
	\logdet \cA^\dagger(Y^0) \,-\, \left<G, Y_2^0\right> \,+\, \left<G, Y_2^0\right>
	\\[.15cm]
	&\!\! \geq \!\!&
	\logdet \cA^\dagger(Y^0) \,-\, \left<G, Y_2^0\right>  
	\,-\,
	\gamma\, \sqrt{n}\, \sigma_{\max} (\cA_1^\dagger)\, \trace (\bar{X}).
	\ea
\ee
gives $\lambda_{\min}(\cA^\dagger(Y^0))=a\ge\alpha$.
}

{
The proof of $\alpha I\preceq \cA^\dagger (Y^k)\preceq \beta I$ for $k>0$ follows similar lines. 
We use inductive argument to prove this. Assume that $\alpha I\preceq \cA^\dagger (Y^\ell)\preceq \beta I$ holds for all $0 \leq \ell \leq k-1$. 
For the upper bound, we have
\[
	\ba{rcl}
	\norm{\cA^\dagger(Y^{k})}_2
	&\!\! = \!\!&
	\norm{\cA^\dagger(Y^{k})}_2 \;-\; \norm{\cA^\dagger(\bar{Y})}_2 
	\;+\; 
	\norm{\cA^\dagger(\bar{Y})}_2
	\\[.15cm]
	&\!\! \leq \!\!&
	\norm{\cA^\dagger(Y^{k}) \,-\, \cA^\dagger(\bar{Y})}_F
	\;+\; 
	\norm{\cA^\dagger(\bar{Y})}_2
	\\[.15cm]
	&\!\! \leq \!\!&
	\sigma_{\max}(\cA^\dagger)\, \norm{Y^{k} -\, \bar{Y}}_F
	\;+\; 
	\norm{\cA^\dagger(\bar{Y})}_2.
	\ea
\]
By repeatedly applying Lemma~\ref{lemma:contractive} for $\norm{Y^{j} -\, \bar{Y}}_F \leq \norm{Y^{j-1} -\, \bar{Y}}_F$, for all $1 \leq j \leq k$ we have
\begin{align}
\non
	\norm{\cA^\dagger(Y^{k})}_2
	\;\leq\;
	\sigma_{\max}(\cA^\dagger)\, \norm{Y^0 -\, \bar{Y}}_F
	\;+\; 
	\norm{\cA^\dagger(\bar{Y})}_2
	\;=\;
	\beta.
\end{align}
To see the lower bound, note that \eqref{eq.innerproduct_bound} holds for all $k\ge 0$, namely,
\be
	\label{eq.innerproduct_bound1}
	\ba{rcl}
	\left< G, Y_2^k \right>
	\; = \;
	\left< \cA_2(\bar{X}), Y_2^k \right>
	&\!\! = \!\!&
	\left< \bar{X}, \cA_2^\dagger(Y_2^k) \right>
	\\[.15cm]
	&\!\! \geq \!\!&
	-\, \gamma\, \sqrt{n}\, \sigma_{\max} (\cA_1^\dagger)\, \trace (\bar{X}),
	\ea
\ee
with the same argument. Let $a=\lambda_{\min}(\cA^\dagger(Y^k))$. Since $\beta$ gives a bound on the largest eigenvalue of $\cA^\dagger(Y^k)$, we have
\be
	\label{eq.logdet_bound1}
	\logdet \cA^\dagger (Y^k)
	\;\leq\;
	\log(a)
	\,+\,
	(n - 1)\log(\beta)
\ee
The dual ascent property
\[
	\logdet \cA^\dagger(Y^k) \,-\, \left<G, Y_2^k\right>
	\; \geq \;
	\logdet \cA^\dagger(Y^0) \,-\, \left<G, Y_2^0\right>,
\]
together with inequality~\eqref{eq.innerproduct_bound1} gives
\be
	\label{eq.logdet_bound2}
	\ba{rcl}
	\logdet \cA^\dagger(Y^k)
	&\!\! \geq \!\!&
	\logdet \cA^\dagger(Y^0) \;-\, \left<G, Y_2^0\right> \,+\, \left<G, Y_2^k\right>
	\\[.15cm]
	&\!\! \geq \!\!&
	\logdet \cA^\dagger(Y^0) \;-\, \left<G, Y_2^0\right>  
	\,-\,
	\gamma\, \sqrt{n}\, \sigma_{\max} (\cA_1^\dagger)\, \trace (\bar{X}).
	\ea
\ee
From~\eqref{eq.logdet_bound1} and~\eqref{eq.logdet_bound2} we thus have
\[
	\lambda_{\min}(\cA^\dagger(Y^k))
	\;=\; 
	a
	\; \geq \;
	\det \cA^\dagger(Y^0)\, \beta^{1-n}\, \mre^{-\left<G, Y_2^0\right> -\, \gamma \sqrt{n}\, \sigma_{\max}(\cA^\dagger_1)\,\trace(\bar{X})}
	\;=\;
	\alpha,
\]
which completes the proof.
}

	\vspace*{-4ex}
\def\IEEEbibitemsep{-5pt plus 0pt}
\setstretch{0.975}

\end{document}